\documentclass[a4paper,12pt]{article} 
\usepackage{graphicx}
\usepackage[english]{babel} 
\usepackage{amsmath,amssymb,amsthm,amsfonts} 
\usepackage{latexsym}
\usepackage[all]{xy}
\usepackage{url}
\usepackage{mathrsfs}
\usepackage[latin1]{inputenc}
\usepackage{tabularx}

\usepackage{hyperref}
\hypersetup{
  pdftitle={Equivariant epsilon constant conjectures for weakly ramified extensions},
  pdfauthor={Werner Bley and Alessandro Cobbe}}

\newcommand{\Q}{\mathbb Q}

\newcommand{\Z}{\mathbb Z}
\newcommand{\C}{\mathbb C}
\newcommand{\R}{\mathbb R}

\newcommand{\p}{\mathfrak p}
\newcommand{\frp}{\mathfrak p}
\newcommand{\gal}{\mathrm{Gal}}
\renewcommand{\epsilon}{\varepsilon}

\newcommand{\Kur}{K_{\mathrm{nr}}}
\newcommand{\Nur}{N_{\mathrm{nr}}}
\newcommand{\oo}{\mathcal O}
\newcommand{\calN}{\mathcal{N}}
\newcommand{\calO}{\mathcal O}
\newcommand{\calT}{\mathcal T}
\newcommand{\Tr}{\mathcal{T}}
\newcommand{\ON}{\mathcal{O}_N}
\newcommand{\OK}{\mathcal{O}_K}
\newcommand{\OL}{\mathcal{O}_L}
\newcommand{\ff}{\mathcal F}
\newcommand{\calL}{\mathcal{L}}
\newcommand{\M}{\mathcal M}
\newcommand{\mm}{\mathfrak M}

\newcommand{\CGa}{\C[\Gamma]}
\newcommand{\ZG}{{\Z[G]}}
\newcommand{\OKG}{\calO_K[G]}
\newcommand{\ZGa}{\Z[\Gamma]}
\newcommand{\Tloc}{T\Omega^\mathrm{loc}(F/E, 1)}
\newcommand{\Irr}{\mathrm{Irr}}
\newcommand{\ZpG}{\Z_p[G]}
\newcommand{\QpG}{\Q_p[G]}
\newcommand{\Qp}{{\Q_p}}
\newcommand{\Zp}{{\Z_p}}
\newcommand{\Qpc}{\Q_p^c}
\newcommand{\lra}{\longrightarrow}
\newcommand{\Opt}{\calO_p^t}
\newcommand{\sseq}{\subseteq}
\newcommand{\Ext}{\mathrm{Ext}}
\newcommand{\Hom}{\mathrm{Hom}}
\newcommand{\Kappa}{\mathcal{K}}

\newtheorem{teo_intro}{Theorem}
\newtheorem{coroll_intro}{Corollary}

\newtheorem{teo}{Theorem}[subsection]
\newtheorem{lemma}[teo]{Lemma}

\newtheorem{prop}[teo]{Proposition}

\theoremstyle{remark}
\newtheorem{remark_intro}{Remark}
\newtheorem{conjecture_intro}{Conjecture}

\theoremstyle{definition}

\newtheorem*{property}{Property}

\title{Equivariant epsilon constant conjectures\\for weakly ramified extensions}
\author{Werner Bley and Alessandro Cobbe}
\date{}
\begin{document}
\maketitle
\begin{abstract}
We study the local epsilon constant conjecture as formulated by Breuning in \cite{Breuning04}. This conjecture fits into the
general framework of the equivariant Tamagawa number conjecture (ETNC) and should be interpreted as a consequence of the expected compatibility of
the ETNC with the functional equation of Artin-$L$-functions. 

Let $K/\Q_p$ be unramified. Under some mild technical assumption we prove Breuning's conjecture for weakly ramified abelian extensions $N/K$ with cyclic ramification group. 
As a consequence of Breuning's local-global principle we obtain the validity of the global
epsilon constant conjecture as formulated in  \cite{BlBu03} and of Chinburg's $\Omega(2)$-conjecture as stated in \cite{Chinburg85}
for certain infinite families $F/E$ of weakly and wildly ramified extensions of number fields.
\end{abstract}

\begin{section}{Introduction}
We fix a Galois extension $F/E$ of number fields and set $\Gamma := \gal(F/E)$. Let $S$ be a sufficiently large finite set
of places of $E$ which, in particular, includes all archimedean places and all places which ramify in $F/E$. Let $\zeta_{F/E, S}(s)$
denote the $S$-truncated equivariant zeta-function of $F/E$ as defined in \cite[Sec.~2.3]{BrBu07} which takes values in the centre $Z(\CGa)$ of the complex group ring $\CGa$.
We recall that $\zeta_{F/E, S}(s)$ can be considered as the vector consisting of $S$-truncated Artin $L$-functions for all irreducible
characters of $\Gamma$. For each rational integer $m$ we write $\zeta_{F/E,S}^*(m)$ for the leading non-zero coefficient in the Taylor expansion
of $\zeta_{F/E, S}(s)$ at $s=m$. It follows easily that $\zeta_{F/E,S}^*(m)$ is contained in the unit group of $Z(\R [\Gamma])$ 
(cf. \cite[Lemma 2.7]{BrBu07}). 

Continuing work of Burns in \cite{Bur01}, Breuning and Burns 
formulate in \cite{BrBu07} explicit conjectures for the image of $\zeta_{F/E,S}^*(0)$, resp.  $\zeta_{F/E,S}^*(1)$, under the canonical homomorphism
$\hat\partial$ from $Z(\R [\Gamma])$ to the relative algebraic $K$-group $K_0(\ZGa, \R)$. We recall that the conjectural formula for $\zeta_{F/E,S}^*(0)$ 
is equivalent to the \emph{lifted root number conjecture} of Gruenberg, Ritter and Weiss (cf. \cite{GRW}), and moreover, is expected to be equivalent to the 
equivariant Tamagawa number conjecture for the pair $(h^0(\mathrm{Spec}(F)), \ZGa)$ (cf. \cite[Prop.~4.4 and Rem.~4.5]{BrBu07}). Under some technical
hypothesis  the conjectural formula for $\zeta_{F/E,S}^*(1)$ is shown in \cite[Th.~1.1 and Cor.~1.2]{BrBu10} to be equivalent to the equivariant Tamagawa number conjecture
as applied to the pair  $(h^0(\mathrm{Spec}(F))(1), \ZGa)$.

In this paper we provide new evidence for the functional equation compatibility of these conjectures.
To be more specific, we recall that Breuning and Burns define elements $T\Omega(F/E, m)$ in  $K_0(\ZGa, \R)$ for $m = 0,1$  and state their conjectures 
in the form $T\Omega(F/E, m) = 0$ (cf. \cite[Conj.~3.3 and 4.1]{BrBu07}). Motivated by the work in \cite{BlBu03} they define a further element
$\Tloc$ in $K_0(\ZGa, \R)$ and show in \cite[Th.~5.2]{BrBu07} that
\[
\psi_\Gamma^*\left( T\Omega(F/E, 0) \right) -  T\Omega(F/E, 1) = \Tloc.
\]
Here $\psi_\Gamma^*$ denotes a natural involution on the algebraic $K$-group $K_0(\ZGa, \R)$.

The leading term conjectures for $\zeta_{F/E,S}^*(0)$ and  $\zeta_{F/E,S}^*(1)$ force the following conjecture which we want to study in this paper.

\begin{conjecture_intro}\label{global eps conj} (cf. \cite[Conj.~5.3]{BrBu07})
One has the equality
\[
\Tloc = 0 
\]
in  $K_0(\ZGa, \R)$.
\end{conjecture_intro}

By \cite[Rem.~5.4]{BlBu03} Conjecture  \ref{global eps conj} is equivalent to Conjecture 4.1 of \cite{BlBu03}.
We recall that for every Galois extension $F/E$ the invariant $\Tloc$ lies in the finite group $K_0(\ZGa, \Q)_{\mathrm{tors}}$, the torsion subgroup of
$K_0(\ZGa, \Q) \sseq K_0(\ZGa, \R)$ (\cite[Cor.~6.3 (i)]{BlBu03}). Moreover, 
Conjecture \ref{global eps conj} is known if $F/E$ is at most tamely ramified (\cite[Cor.~7.7]{BlBu03}), if 
$F$ is an abelian extension of $\Q$ with odd conductor (\cite[Cor.~5.4 (ii)]{BlBu03}) or if $F$ is an extension of $\Q$ of degree $\le 15$  (\cite[Cor.~7]{BlDe13}).
We also recall that by \cite[Rem.~4.2 (iv)]{BlBu03} Conjecture  \ref{global eps conj} implies Chinburg's $\Omega(2)$-conjecture as stated in  \cite{Chinburg85}.

Conjecture  \ref{global eps conj} is essentially of local nature. In fact, it is a local approach which lies behind the proofs of the known
cases mentioned above. Based on this observation, Breuning  stated in \cite{Breuning04} an independent conjecture for Galois extensions $N/K$ of local number fields.
We write $G$ for the Galois group of $N/K$.
Breuning defined an element $R_{N/K}$ in $K_0(\Z_p[G], \Q_p)$ incorporating local epsilon constants and algebraic invariants associated
to $N/K$. We will briefly recall the definition of $R_{N/K}$ in Section \ref{Breunings conjecture}. Breuning stated the following conjecture.

\begin{conjecture_intro}\label{local eps conj} (cf. \cite[Conj.~3.2]{Breuning04})
One has the equality
\[
R_{N/K} = 0 
\]
in  $K_0(\Z_p[G], \Q_p)$.
\end{conjecture_intro}
Since $\Tloc$ is contained in the subgroup $K_0(\ZGa, \Q)$ it can be studied prime by prime.
We let $\Tloc_p \in  K_0(\Z_p[\Gamma], \Q_p)$ denote its $p$-primary part.
Then the local conjecture is related to the global conjecture by the equation
\[
\Tloc_p = \sum_v i_{\Gamma_w}^\Gamma \left( R_{F_w / E_v} \right),
\]
where $v$ runs through all places of $E$ above $p$, $w$ is a fixed place of $F$ lying over $v$, $\Gamma_w$ denotes the decomposition group and
$i_{\Gamma_w}^\Gamma$ is the induction map on the relative algebraic $K$-group, cf. \cite[Th.~4.1]{Breuning04}.

In \cite{Breuning04,Breuning04b} Breuning proved Conjecture \ref{local eps conj} for tamely ramified extensions, for abelian
extensions of $\Q_p$ with $p \ne 2$, for all $S_3$-extensions of $\Q_p$, and for certain families of dihedral and quaternion extensions.
If $p$ is odd,  an algorithmic proof for Conjecture \ref{local eps conj} is given in \cite{BlDe13} for all Galois extensions of degree $\le 15$. If $p=2$, the conjecture is
also proved in loc.cit. for all non-abelian Galois extensions of $\Q_2$ with $[N : \Q_2] \le 15$ and, in addition, for all abelian extensions $N/\Q_2$ with
$[N : \Q_2] \le 7$.

In this manuscript we will focus on weakly and wildly ramified extensions $N/K$ of an unramified extension $K/\Q_p$. We recall that $N/K$ is weakly ramified, if the second
ramification group in lower numbering is trivial.

We state the main result of our work.

\begin{teo_intro}\label{main theorem}
Let $p$ be an odd prime and let $K/\Q_p$ be a finite unramified extension. Let $m$ denote the degree of $K/\Q_p$.
Let $N/K$ be a weakly and wildly ramified finite abelian extension with cyclic ramification group. Let $d$ denote the inertia degree of
$N/K$ and assume that $m$ and $d$ are relatively prime. Then Conjecture \ref{local eps conj} is true for $N/K$.
\end{teo_intro}

\begin{remark_intro}
The assumptions of the theorem imply that the ramification group is cyclic of order $p$ (cf. \cite[Cor.~3.4]{PickettVinatier}). 
More precisely, $|G| = pd, |G_0| = |G_1| = p$ and $|G_i| = 1$ for $i \ge 2$. Here $G_i$ for $i \ge 0$ denotes the higher ramification subgroup.
\end{remark_intro}

The invariant $R_{N/K}$ incorporates amongst other terms the equivariant local epsilon constant and a certain equivariant
discriminant attached to $N/K$. Whereas the main ingredient in the definition of the equivariant epsilon constant is a local 
Gau\ss\ sum, equivariant discriminants are closely related to norm-resolvents. The relation between norm-resolvents and Galois Gau\ss\ sums
in the context of Theorem \ref{main theorem} is analyzed by Pickett and Vinatier in \cite{PickettVinatier}. Indeed, Theorem 2 of
loc.cit. was one of the main motivations and a starting point for our project. In addition, the strategy for the proof of 
Theorem \ref{main theorem} was inspired by the reductions 
made in Section 3.3 of loc.cit.

The above relation between
$T\Omega^{\rm loc}(F/E,1)_p$ and $R_{F_w/E_v}$ implies results for global Galois extensions $F/E$
which satisfy the following property. 

\begin{property}[$\bf{*}$]
We say that the Galois extension $F/E$ of number fields satisfies Property $(*)$ if for every wildly
ramified place $v$ of $E$ with $w | v | p$ one of the following cases is satisfied
\begin{itemize}
\item[a)] $E_v = \Q_p$, $p>2$ and $\Gamma_w$ is abelian,
\item[b)] $E_v = \Q_p$, $p=2$, $\Gamma_w$ is abelian and $|\Gamma_w|\le 7$,
\item[c)] $E_v = \Q_p$, $p\ge2$, $\Gamma_w$ is non-abelian and $|\Gamma_w|\le 15$,
\item[d)] $E_v/\Q_p$ is unramified, $p > 2$, $F_w/E_v$ is abelian and weakly ramified with 
cyclic ramification group and $[E_v:\Q_p]$ is coprime with the inertia degree of $F_w / E_v$.
\end{itemize}
\end{property}

Every tamely ramified extension $F/E$  obviously satisfies Property $(*)$.
It is easy to construct infinite families of weakly and wildly ramified extensions of number fields which satisfy condition d) using class field theory. 
In particular, 
if $p$ is an odd prime, $E/\Q$ an extension of number fields in which $p$ is unramified and $F/E$ a cyclic extension of degree $p$ which is at most
weakly ramified, then $F/E$ satisfies Property $(*)$.

\begin{coroll_intro}
\label{cor:globeps}
Conjecture \ref{global eps conj} is valid for all
Galois extensions $F/E$ which satisfy Property $(*)$.
\end{coroll_intro}

The projection onto the class group also proves
Chinburg's conjecture:

\begin{coroll_intro}
\label{cor:omega2}
Chinburg's $\Omega(2)$-conjecture is valid for all
Galois extensions $F/E$  which satisfy Property $(*)$.
\end{coroll_intro}

Moreover, the functorial properties of \cite[Prop.~3.3]{Breuning04}
imply the following result:

\begin{coroll_intro}
\label{cor:relative}
Conjecture \ref{global eps conj}  and Chinburg's $\Omega(2)$-conjecture
are valid for global Galois extensions $F/E$
for which $E' \subseteq E \subseteq F \subseteq F'$ with a Galois extension $F'/E'$  that satisfies Property~$(*)$.
\end{coroll_intro}

In Section \ref{Breunings conjecture} we will first recall Breuning's conjecture and then give a short description of the organization of
the manuscript.

{\bf Notations} Given a field extension $F/E$, we will denote the norm and the trace  by $\mathcal N_{F/E}$ and $\mathcal T_{F/E}$ respectively. 
If $K$ is a local field, then $v_K$ will always denote its normalized valuation. We will write $\OK$ and $\frp_K$ for the valuation ring and
the maximal ideal respectively.  Furthermore,  $U_K$ will denote the units of $\OK$ and $U_K^{(n)}:=\{u\in U_K: u\equiv 1\pmod{\p_K^n}\}$ the higher principal units.

If $K$ is a field we write $K^c$ for an algebraic closure. For a finite group $G$ we write $\Irr_{\Q^c}(G)$ for the set of 
absolutely irreducible characters of $G$. We often implicitly fix an embedding $\Q^c \hookrightarrow \Qpc$ and view $\Q^c$-valued characters as
valued in $\Qpc$.
 
If $H \le G$ is a subgroup, then  $e_H=\frac{1}{|H|}\sum_{\sigma\in H}\sigma$ denotes the usual subgroup idempotent. We also set $T_H := |H|e_H$.
For $a\in G$ we abbreviate $e_a=e_{\langle a\rangle}$ and $T_a=T_{\langle a\rangle}$.

For a $\Z$-module $M$ and a prime $p$ we often write $M_p$ for $M \otimes_\Z \Zp$.

\end{section}

\begin{section}{The local epsilon constant conjecture}\label{Breunings conjecture}
In this section we briefly recall the formulation of Breuning's local epsilon constant conjecture.
For further details we refer the reader to \cite[Sec.~2]{Breuning04}. 

\begin{subsection}{The shape of the conjecture}\label{shape}

The element $R_{N/K}$ is of the form
\[
R_{N/K}=T_{N/K}+C_{N/K}+U_{N/K}-M_{N/K}
\]
where each of the terms is an element in $K_0(\ZpG, \Qpc)$. This algebraic $K$-group lies in an exact localization sequence
of the form
\[
K_1(\ZpG) \lra K_1(\Qpc[G]) \lra  K_0(\ZpG, \Qpc) \lra 0.
\]
If $G$ is abelian, the determinant induces an isomorphism $K_1(\Qpc[G]) \simeq \Qpc[G]^\times$. Since $\ZpG$ is semilocal the natural
map $\ZpG^\times \lra K_1(\ZpG)$ is onto, so that in the abelian case we can and will identify $K_0(\ZpG, \Qpc)$ with $ \Qpc[G]^\times / \ZpG^\times$.
Furthermore, we identify $\Qpc[G]^\times$ with $\prod_\chi \left( \Qpc \right) ^\times$ where $\chi$ runs through the
set $\Irr_{\Q^c}(G)$.

The term $T_{N/K}$ is called the \emph{equivariant local epsilon constant}. If $K$ is a finite extension of $\Q_p$ and $\chi$ a character of
$\gal(K^c/K)$ with values in $\Q^c$ we write $\tau_K(\chi) \in \Q^c$ for the local 
Galois Gau\ss\ sum as defined in \cite[II, Sec.~4]{Martinet77}.
Let $N/K$ be an abelian extension of $p$-adic fields and put $G := \gal(N/K)$. We set
\[
\tau_{N/K} := \left( \tau_{\Q_p}\left(i_K^\Qp\chi \right) \right)_{\chi \in \Irr_{\Q^c}(G)} \in \prod_\chi \left( \Q^c \right)^\times = \Q^c[G]^\times.
\]
Let $k \colon \Q^c \lra \Qpc$ be any embedding and also write $k \colon \Q^c[G] \lra \Qpc[G]$ for the induced map.
Then $T_{N/K} \in  \Qpc[G]^\times / \ZpG^\times$ is defined to be the class represented by $k(\tau_{N/K})$. By \cite[Lemma 2.2]{Breuning04} 
the definition  $T_{N/K}$ does not depend on the choice of the embedding $k$. 

We call $C_{N/K}$ the \emph{cohomological term}. Let $\calL$ be a full projective $\ZpG$-sublattice of $\ON$ which is contained in
a sufficiently high power of the maximal ideal such that the exponential map of $N$ is defined on $\calL$. We recall that in
\cite[Sec.~3.3]{BlBu03} a cohomological term $E(X) \in K_0(\Z[G], \Q)$ is defined for every cohomologically trivial
$\Z[G]$-submodule $X$ of finite index in $U_N$. We write $E(X)_p \in K_0(\Z_p[G], \Q_p)$ for its $p$-part. Then,
by \cite[Prop.~2.6]{Breuning04},
\[
C_{N/K}=E(\exp(\mathcal L))_p-[\mathcal L,\rho_N,H_N]
\]
in $K_0(\Z_p[G], \Q^c_p)$. The computation of $E(\exp(\calL))_p$ in our special situation is the technical heart of this paper.
We therefore postpone its definition to Section \ref{definition of EX}. For the definition of 
$[\mathcal L,\rho_N,H_N]$ we just recall that for a normal basis element $\theta \in \ON$ and $\calL := \OKG\theta$ the element 
$[\mathcal L,\rho_N,H_N]$ is represented by 
$\left( \delta_K \calN_{K/\Qp}(\theta \mid \chi) \right)_{\chi \in \Irr_{\Q^c}(G)} \in \prod_\chi \left( \Q^c \right)^\times$
where $\calN_{K/\Qp}(\theta \mid \chi)$ denotes the norm resolvent and $\delta_K$ is a root of the discriminant of $K$ (cf.
\cite[Lemma 2.7]{Breuning04}).

We continue to describe the \emph{correction term} $M_{N/K}$. To simplify the notation we assume that $G$ is abelian.
For $x \in \QpG$ we define an invertible element ${}^*x \in \QpG^\times$ as follows.
If $\QpG = \prod F_i$ is the Wedderburn decomposition of $\QpG$ and $x = (x_i)$ under this decomposition, then ${}^*x = \left( ({}^*x)_i \right)$
with  $({}^*x)_i = x_i$ if $x_i \ne 0$ and $({}^*x)_i = 1$ if $x_i = 0$. Let $I$ be the ramification group of $G$ and let
$\sigma \in G$ be a lift of the Frobenius automorphism in $G/I$. Put $q := | \OK / \frp_K |$. Then $M_{N/K} \in K_0(\ZpG, \Q_p)$ is represented by
\[
m_{N/K} := \frac{{}^*(|G/I|e_G){}^*((1-\sigma q^{-1})e_I)}{{}^*((1-\sigma^{-1})e_I)}.
\]

Finally we discuss the \emph{unramified term} $U_{N/K}$. We write $\Opt$ for the ring of integers in the maximal tamely ramified
extension of $\Qp$ in $\Qpc$. Let $\iota \colon K_0(\ZpG, \Qpc) \lra K_0(\Opt[G], \Qpc)$ be the natural scalar extension map.
We recall that by Taylor's fixed point theorem the restriction of $\iota$ to the subgroup $K_0(\ZpG, \Q_p)$ is injective. If $G$ is abelian,
this injectivity is equivalent to
\[
\QpG^\times / \ZpG^\times \hookrightarrow \Qpc[G]^\times / \Opt[G]^\times.
\]
By \cite[Prop.~2.12]{Breuning04} we have $\iota(U_{N/K}) = 0$. The properties of $U_{N/K}$ with respect to the action
of $\gal(\Qpc / \Q_p)$ ensure that $R_{N/K} \in K_0(\ZpG, \Q_p)$. By Taylor's fixed point theorem it therefore suffices to show that $\iota(R_{N/K}) = 0$
(cf. \cite[Cor.~3.5]{Breuning04}). In the abelian case it therefore suffices to prove that a representative of
$T_{N/K} + C_{N/K} -  M_{N/K}$ actually lies in $\Opt[G]^\times$.

\end{subsection}

\begin{subsection}{Definition of $E(X)$}\label{definition of EX}

Let $N/K$ be a finite Galois extension of $p$-adic fields with group $G$.
Let $X \sseq U_N$ be any cohomologically trivial $\ZG$-submodule of finite index. The element $E(X) \in K_0(\ZG, \Q)$ is defined in
\cite[(19)]{BlBu03}. We recall here the approach summarized in \cite[Lemma 3.7]{BlBu03} which allows an explicit description of $E(X)_p$. This
approach is based on the observation of Burns and Flach made in \cite[Prop.~3.5 (a)]{BurnsFlach98} that relates certain complexes arising from the
cohomology of the sheaf $\mathbb{G}_m$ to $2$-extensions representing the fundamental class of local class field theory.

We fix a $\ZG$-equivariant resolution of $\Z$ of the form
\begin{equation*}\label{resol}
 0 \lra \Sigma \stackrel{\subset}\lra \ZG^r \stackrel{d_2}\lra \ZG \stackrel{d_1}\lra \Z \lra 0 
\end{equation*} 
where $\Sigma := \ker(d_2)$ and compute groups of the form
$\Ext_\ZG^2(\Z , -)$ with respect to this resolution.
We then choose a morphism $\varphi \in \Hom_\ZG(\Sigma, N^\times / X )$ which
represents the image of the local fundamental class under the canonical
isomorphism  $\Ext_\ZG^2(\Z, N^\times) \lra \Ext_\ZG^2(\Z, N^\times / X )$. Without loss of generality we may assume that $\varphi$ is
surjective. We then set $B := \ker(d_1)$ and $\Kappa :=
\ker(\varphi)$ and we write $i_1, i_2$ and $i_3$ for the inclusion
morphisms
 $\Kappa_\Q \stackrel{\subset}\lra \Sigma_\Q$, $\Sigma_\Q \stackrel{\subset}\lra
  \Q [G]^r$ and $B_\Q \stackrel{\subset}\lra \Q [G]$ respectively.
   We also choose $\Q [G]$-equivariant sections $\rho, \sigma$ and $\tau$ to the
morphisms $\varphi_\Q, d_{2,\Q}: \Q [G]^r \lra B_\Q$ and
 $d_{1,\Q}$ respectively. We then write $\tilde\theta$ for the composite
isomorphism
\begin{eqnarray}\label{thetatilde}
(\Kappa \oplus \ZG)_\Q &\stackrel{(\mathrm{id},
(\tau,i_3)^{-1})}\longrightarrow& \Kappa_\Q \oplus (\Q \oplus
B_\Q) \nonumber
\\ &\stackrel{(\mathrm{id},\nu_N^{-1},\mathrm{id})}\lra& \Kappa_\Q \oplus \left(N^\times/X\right)_\Q\oplus B_\Q
 \nonumber \\ &\stackrel{(i_1 ,\rho, \mathrm{id})}\lra&
\Sigma_\Q \oplus B_\Q \label{theta tilde} \\ &\stackrel{(i_2,
\sigma)}\lra& \Q[G]^r. \nonumber
\end{eqnarray}
By \cite[Lemma 3.7]{BlBu03} the module $\Kappa$ is finitely generated and  $\ZG$-projective and, moreover, $$E(X) = [\Kappa \oplus \ZG,
\tilde\theta, \ZG^r]$$ in $K_0(\ZG, \Q)$ .

Suppose now that $G$ is abelian. In order to compute a representative of $E(X)_p$ in $\QpG^\times / \ZpG^\times \simeq K_0(\ZpG, \Q_p)$ 
we first note that $\Kappa_p \oplus \ZpG$ is $\ZpG$-free. We choose $\ZpG$-bases of $\Kappa_p \oplus \ZpG$ and $\ZpG^r$, respectively, 
and let $A_{\tilde\theta} \in \mathrm{Gl}_r(\QpG)$ denote the matrix which represents $\tilde\theta$ with respect to this choice of bases.
Then $E(X)_p$ is represented by $\det(A_{\tilde\theta})$.

\end{subsection}

\begin{subsection}{Plan of the manuscript}\label{plan}

In Section \ref{heart} we will compute the term $E(\exp(\calL))_p$ for extensions $N/K$ as in Theorem \ref{main theorem} and a special choice of
lattice $\calL$. The term $[\calL, \rho_N, H_N] - T_{N/K}$ is represented by the quotient of a norm resolvent by Galois Gau\ss\ sums.
In Section \ref{section_normresolvent_gauss} we will use the main result of \cite{PickettVinatier} to quickly
compute this term. Finally in Section \ref{section_the_proof} we calculate $M_{N/K}$ and complete the proof of Theorem \ref{main theorem} by showing 
that a representative of $T_{N/K} + C_{N/K} -  M_{N/K}$ lies in $\Opt[G]^\times$.
 \end{subsection}

\end{section}

\begin{section}{The setting}\label{section setting 1}
\begin{subsection}{Definitions and notation}\label{subsection def und not}
In this section we fix the setting in which we will work. We will consider local field extensions as follows.
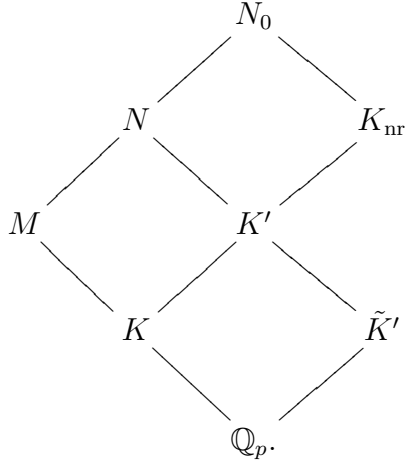
\[\xymatrix{&&N_0\ar@{-}[dr]\ar@{-}[dl]\\&N\ar@{-}[dr]\ar@{-}[dl]&&\Kur\ar@{-}[dl]\\M\ar@{-}[dr]&&K'\ar@{-}[dl]\ar@{-}[dr]\\&K\ar@{-}[dr]&&\tilde K'\ar@{-}[dl]\\&&\Q_p.}\]
Here $K/\Q_p$ is the unramified extension of degree $m$ and $K'/K$ is the unramified extension of degree $d$.
We assume throughout that $(m,d)=1$. Furthermore, $M/K$ is a weakly and wildly ramified cyclic extension of degree $p$. 
Since $(m,d)=1$, there exists $\tilde K'/\Q_p$ of degree $d$ such that $K'=K\tilde K'$.
Further we set $N=MK'$, $\Kur$ the maximal unramified extension of $K$ and $N_0=\Kur N$. 
We will prove Conjecture \ref{local eps conj} for the extension  $N/K$. 

Let $F\in\gal(N_0/M)\cong \gal(\Kur/K)$ be the Frobenius automorphism, let $F_0=F^d\in\gal(N_0/N)\cong \gal(\Kur/K')$ and put $q=p^m$. We consider elements $a,b\in \gal(N_0/K)$ such that $\gal(M/K)=\langle a|_M\rangle$, $a|_{\Kur}=1$, $b|_M=1$ and $b|_{\Kur}=F^{-1}$. Since there will be no ambiguity, we will denote by the same letters $a,b$ their restrictions to $N$. Then $\gal(N/K) = \langle a,b \rangle$ and $\mathrm{ord}(a)=p, \mathrm{ord}(b) = d$.

\begin{lemma}\label{defnA}
Let $L/k$ be a finite tamely ramified Galois extension of $p$-adic fields. Then there exists a normal integral basis generator of trace one. 
\end{lemma}

\begin{proof}
Put $\Delta := \gal(L/k)$. By Noether's Theorem there exists an element $\theta \in \OL$ such that $\OL = \calO_k[\Delta]\theta$. Let $t := \calT_{L/k}(\theta)$.
Then $t \in \calO_k^\times$ and $\frac{\theta}{t}$ is an integral normal basis generator of trace one.
\end{proof}

Let us call $A$ such an element for the extension $K/\Q_p$ and let $\theta_2$ be such an element for the extension $\tilde K'/\Q_p$.

\begin{lemma}\label{defntheta1}
There exists an element $\theta_1\in \p_M$ such that $\oo_K[\gal(M/K)]\theta_1=\p_M$ and we can assume that $\Tr_{M/K}\theta_1=p$.
\end{lemma}

\begin{proof}
By \cite[Th.~1.1 and Lemma 1.4 (b)]{Koeck} there exists an element $\tilde\theta_1 \in \p_M$ such that 
$\OK[\gal(M/K)]\tilde\theta_1 = \p_M$ and $\Tr_{M/K}(\tilde\theta_1) = up$
with a unit $u \in \calO_K^\times$. We set $\theta_1 := \frac 1 u \tilde\theta_1$. 
\end{proof}

For the rest of the paper we fix an element $\theta_1 \in \p_M$ as in Lemma \ref{defntheta1}. 
Since $a\in G_1\setminus G_2$, where $G_i$ is the $i$-th ramification group of $G=\gal(N/K)$, 
we know by \cite[Sec.~IV.2, Prop.~5]{SerreLocalFields} that $\theta_1^{a-1}\equiv 1-\alpha_1\theta_1\pmod{\p_M^2}$ for some unit $\alpha_1\in\oo_M^\times$. Since $\alpha_1$ can be replaced by any element in the same residue class in $\oo_M/\p_M=\oo_K/\p_K$, we can assume that $\alpha_1\in\oo_K^\times$.

By our choice of $A$, we know that $A,A^{f},\dots A^{f^{m-1}}$ is a basis of $\oo_K$ over $\Z_p$, where $f$ denotes the Frobenius automorphism of $\Kur/\Q_p$. 
Since $1=\Tr_{K/\Q_p}A=\sum_{i=0}^{m-1}A^{f^i}$ and $\alpha_1\in\oo_K^\times$ it easily follows that also 
\begin{equation}\label{choice of alpha}
\alpha_1,\alpha_2=\alpha_1A,\alpha_3=\alpha_1A^{f},\dots, \alpha_m=\alpha_1A^{f^{m-2}}
\end{equation}
constitute a basis of $\oo_K$ over $\Z_p$. In particular, we have the equality $A=\frac{\alpha_2}{\alpha_1}$.

\begin{lemma}\label{xp-x+Atheta2}
With the above notation, $X^p-X+A\theta_2$ divides $X^{q^d}-X+1$ in $\oo_{K'}/\p_{K'}[X]$.
\end{lemma}
\begin{proof}
We have
\[X^{q^d}-X+1= X^{p^{md}}-X+\Tr_{K'/\Q_p}(A\theta_2)\equiv\sum_{i=0}^{md-1}(X^p-X+A\theta_2)^{p^i}\pmod{\p_{K'}}\]
and the right hand side is clearly a multiple of $X^p-X+A\theta_2$.
\end{proof}

Now we choose an element $x_2\in\oo_{\Kur}$ so that the class of $\frac{x_2}{\alpha_1}$ modulo $\p_{\Kur}$ is a root of the polynomial $X^p-X+A\theta_2$.
Let $\zeta_{q^d-1} \in \calO_{K'}^\times$ be a primitive $(q^d-1)$-th root of 
unity.

\begin{lemma}\label{F0-1surjectivemodp2}
We have
\[(\zeta_{q^d-1}(1+x_2\theta_1))^{F_0-1}\equiv\theta_1^{a-1}\equiv 1-\alpha_1\theta_1\pmod{\p_{N_0}^2}.\]
\end{lemma}

\begin{proof}
By the choice of $x_2$ and Lemma \ref{xp-x+Atheta2} we obtain 
\[
\left(\frac{x_2}{\alpha_1}\right)^{q^d}-\frac{x_2}{\alpha_1}\equiv-1\pmod{\p_{\Kur}}.
\]
Multiplying by $\alpha_1^{q^d}$ and observing $\alpha_1^{q^d - 1} \equiv 1 \pmod{\p_{\Kur}}$ we obtain
\begin{equation}
  \label{eq:1000}
  x_2^{q^d}-x_2\equiv-\alpha_1\pmod{\p_{N_0}}.
\end{equation}
Now we conclude
\[
\begin{split}\left(\zeta_{q^d-1}(1+x_2\theta_1)\right)^{F_0-1}
&=(1+x_2\theta_1)^{b^{-d}}(1+x_2\theta_1)^{-1}
\\&\equiv(1+x_2^{b^{-d}}\theta_1)(1-x_2\theta_1)\pmod{\p_{N_0}^2}\\
&\equiv 1+x_2^{b^{-d}}\theta_1-x_2\theta_1\equiv 1-\alpha_1\theta_1\pmod{\p_{N_0}^2},
\end{split}\]
where the last congruence follows from (\ref{eq:1000}).
\end{proof}

Let $\hat N_0$ denote the completion of $N_0$.

\begin{lemma}\label{F0-1surjective}
For all $u\in U_{\hat N_0}$  there exists $z\in U_{\hat N_0}$ such that $z^{F_0-1}=u$. 
In particular, there exists $\gamma\in U_{N_0}$ such that $\gamma^{F_0-1}\equiv \theta_1^{a-1}\pmod{\p_{N_0}^{p+1}}$ and the element $\gamma$ can be chosen so that $\gamma\equiv\zeta_{q^d-1}(1+x_2\theta_1)\pmod{\p_{N_0}^2}$.
\end{lemma}

\begin{proof}
The first part of the lemma is contained in \cite[Sec.~V, Lemma 2.1]{Neukirch92}. The second part follows from the constructions made in the
proof of loc.cit. combined with Lemma \ref{F0-1surjectivemodp2}. For the reader's convenience we carry out the details.

Since the residue field of $N_0$ is algebraically closed, there exists a solution $z_1\in U_{N_0}$ of $z^{F_0}\equiv z^{q^d}\equiv zu\pmod{\p_{N_0}}$.

Now let us assume that $i\geq2$ and that we have an element $z_{i-1} \in U_{N_0}$ such that $z_{i-1}^{F_0-1}\equiv u\pmod{\p_{N_0}^{i-1}}$. 
By assumption, $uz_{i-1}^{1-F_0}-1$ is a multiple of $\theta_1^{i-1}$. So we can find a solution $y_i\in\oo_{N_0}$ of 
\[
X^{F_0}-X-\frac{uz_{i-1}^{1-F_0}-1}{\theta_1^{i-1}} \equiv 0 \pmod{\p_{N_0}}.
\]
Multiplying by $\theta_1^{i-1}$, we get
\[
y_i^{F_0}\theta_1^{i-1}\equiv y_i\theta_1^{i-1}+uz_{i-1}^{1-F_0}-1\pmod{\p_{N_0}^i}.
\]
Now we set $z_i := z_{i-1}(1+y_i\theta_1^{i-1})$ and easily verify that $z_i^{F_0 - 1} \equiv u \pmod{\p_{N_0}^i}$. The $z_i$ form a Cauchy sequence 
which converges to an element $z \in U_{\hat N_0}$ with the requested properties. 

If $u=\theta_1^{a-1}$, then by Lemma \ref{F0-1surjectivemodp2} we can start the construction of the sequence of the $z_i$ 
from the element $z_2=\zeta_{q^d-1}(1+x_2\theta_1)$ and take $\gamma=z_{p+1}$.
\end{proof}

For the rest of the paper we fix an element $\gamma \in U_{N_0}$ as in Lemma \ref{F0-1surjective}.

\end{subsection}
\begin{subsection}{Some preliminary results}

In this subsection we collect some preliminary results which will be needed in Section \ref{heart}. We assume all the notations introduced
in Subsection \ref{subsection def und not}.

\begin{lemma}\label{nicenorm}
We have
\[\mathcal N_{M/K}(1-\alpha_1\theta_1)\equiv 1\pmod{\p_{M}^{p+1}}.\]
\end{lemma}

\begin{proof}
Recalling that by \cite[Sec.~V.6, Prop.~8]{SerreLocalFields} $\mathcal N_{M/K}U_{M}^{(2)}\subseteq U_{K}^{(2)}\subseteq U_{M}^{(p+1)}$, we obtain
\[
\mathcal N_{M/K}(1-\alpha_1\theta_1)\equiv \mathcal N_{M/K}(\theta_1^{a-1}) = 1 \pmod{\p_{M}^{p+1}}.
\]
\end{proof}

\begin{lemma}\label{normtrace}
We have
\[\mathcal N_{M/K}(\theta_1)\equiv -\alpha_1^{1-p}p\pmod{\p_M^{p+1}}.\]
\end{lemma}
\begin{proof}
By \cite[Sec.~V.3, Lemma 4]{SerreLocalFields} we have $\Tr_{M/K}(\p_M^2) = \p_K^2$. In addition,
by \cite[Sec.~V.3, Lemma 5]{SerreLocalFields} and Lemma \ref{nicenorm} we obtain
\[
1\equiv\mathcal N_{M/K}(1-\alpha_1\theta_1)\equiv 1+\mathcal N_{M/K}(-\alpha_1\theta_1)+\Tr_{M/K}(-\alpha_1\theta_1)\pmod{\p_M^{p+1}}.
\]
Since $\alpha_1 \in \calO_K^\times$ and $\Tr_{M/K}(\theta_1) = p$ the result easily follows.
\end{proof}

\begin{lemma}\label{Ta_p}
We have
\[p- T_a=(a-1)^{p-1}u,\]
for some unit $u$ of $\Z_p[a]$ such that the augmentation $\epsilon(u)=(p-1)!$.
\end{lemma}

\begin{proof}
One can take, for example, $u = \prod_{i=1}^{p-1} \frac{a^i - 1}{a - 1}$.
\end{proof}

\begin{lemma}\label{norm_power_a-1}
The element $(a-1)^{p-1}-T_a$ is a multiple of $p$ in $\Z_p[a]$. In particular,
\[(a-1)^{p-1}\theta_1\equiv T_a\theta_1 = p\pmod{\p_N^{p+1}}.\]
\end{lemma}

\begin{proof}
Easy exercise.
\end{proof}

\begin{lemma}\label{lemma1Bley}
Let $x\in N^\times$ such that $v_N(x)\in\{1,2,\dots,p-1\}$. Then $v_N((a-1)x)=v_N(x)+1$.
\end{lemma}

\begin{proof}
We have to show that $v_N(x^{a-1} - 1) = 1$ which is equivalent to $x^{a-1} \in U_N^{(1)} \setminus  U_N^{(2)}$.
By our assumptions we have $a\in G_1\setminus G_2$. If $v_N(x) = 1$, then $\ON = \calO_{K'}[x]$ by \cite[Sec.~I.6, Prop.~18]{SerreLocalFields}
and, furthermore,  \cite[Sec.~IV.1, Lemma 1]{SerreLocalFields} implies $x^{a-1} \in U_N^{(1)} \setminus  U_N^{(2)}$. If $v_N(x) \in \{2, \ldots, p-1 \}$,
then we choose $s,t \in \Z$ such that $s v_N(x) + tp =1$. Then $v_N(x^s p^t) = 1$, so that $\left( x^{a-1} \right)^s = \left( x^{s}p^t \right)^{a-1} 
\in U_N^{(1)} \setminus  U_N^{(2)}$. Hence $x^{a-1} \not\in  U_N^{(2)}$, while clearly $x^{a-1}\in U_N^{(1)}$.
\end{proof}

In the following we write $(T_a,(a-1)^j) \sseq \OKG$ for the $\OKG$-submodule generated by $T_a$ and $(a-1)^j$ where $j$ is a non-negative integer.

\begin{lemma}\label{lemma2Bley}
a) Put $\theta := \theta_1 \theta_2$. Then $\p_N = \OKG \theta$.

b) For $j=0,\dots,p-1$ we have $\p_{N}^{j+1}=(p,(a-1)^j)\theta=(T_a,(a-1)^j)\theta$.
\end{lemma}

\begin{proof}
Part (a) is immediate from $\ON = \calO_M \calO_{K'}$ and the definition of $\theta_1$ and $\theta_2$.

By Lemma \ref{Ta_p} we have $(T_a, (a-1)^j) = (p, (a-1)^j)$ which shows the second equality in (b).
Lemma \ref{lemma1Bley} implies the chain of inclusions  $I\theta \subseteq \p_{N}^{j+1}\subseteq \p_{N}= \calO_{K}[G]\theta$, where we have set
$I = I_j := (T_a, (a-1)^j)$. 
Since $\theta$ is a normal basis element we derive
$[\calO_{K}[G]:I]=[\OKG\theta : I\theta]$. So it is enough to show the inequality
$[\calO_{K}[G]:I]\leq [\p_{N}:\p_{N}^{j+1}]$. We observe that
\[
\calO_K[G]=\bigoplus_{i=0}^{p-1}\calO_{K}[b](a-1)^i.
\]
Finally, in order to complete the proof,  we recall that $[\p_{N}:\p_{N}^{j+1}] = q^{dj}$ and note that the  $q^{dj}$ elements in
$\bigoplus_{i=0}^{j-1} (\calO_{K}/\p_{K})[b](a-1)^i$
cover the quotient $\calO_{K}[G]/I$. 
\end{proof}
\end{subsection}

\end{section}

\begin{section}{The computation of $E(\exp(\calL))_p$}\label{heart}

We assume the notations introduced in the previous section. We put $\calL := p\p_N = \p_N^{p+1}$. By Lemma
\ref{defntheta1} $\calL$ is a free $\ZpG$-submodule of $\ON$. Moreover, the exponential function of $N$ is defined on $\calL$
and by \cite[II, Satz (5.5)]{Neukirch92} we have $\exp(\calL) = U_N^{(p+1)}$. In this section we will compute
 a representative in $K_0(\ZpG, \Qp) \simeq \QpG^\times / \ZpG^\times$ for $E(\exp(\calL))_p = E( U_N^{(p+1)} )_p$ as described
at the end of Section \ref{definition of EX}.   

\begin{subsection}{The local fundamental class}

We will need the algebra $\Nur=\Kur\otimes_K N$, on which the group $\gal(\Kur/K)\times G$ acts canonically. 
We obtain an isomorphism  $\Nur \to N_0^d$
by sending $x\otimes y$ to $(F^{d-1}(x)y,F^{d-2}(x)y,\dots,F(x)y,xy)$. 
Then the action of $\gal(\Kur/K)\times G$ on $\Nur$ induces an action on $N_0^d$. 
For later reference we explicitly describe the action for some particular elements (see \cite[Sec.~VI]{Chinburg85}):
\begin{eqnarray}
&& (F^{-1}\times b)(x_1,x_2,\dots,x_d)=(x_1^b,x_2^b,\dots,x_d^b), \nonumber \\
&& (1\times a)(x_1,x_2,\dots,x_d)=(x_1^a,x_2^a,\dots,x_d^a), \label{galois action rules 1} \\
&& (F\times 1)(x_1,x_2,\dots,x_d)=(x_d^{F_0},x_1,x_2,\dots,x_{d-1}). \nonumber
\end{eqnarray}
In particular, we deduce from (\ref{galois action rules 1})
\begin{eqnarray}
&& (1\times b)(x_1,x_2,\dots,x_d) \nonumber \\
&=& (F\times 1)(F^{-1}\times b)(x_1,x_2,\dots,x_d)=((x_d^b)^{F_0},x_1^b,x_2^b,\dots,x_{d-1}^b) \label{galois action rules 2} .
\end{eqnarray}

If $L$ is a field extension of $\Qp$ we put $L(s) := L^\times / U_L^{(s)}$ for each non-negative integer $s$. 
Let $\omega:\Nur^\times\to\Z$ be the sum of the discrete valuations of the different components of $\Nur^\times \simeq (N_0^\times)^{d}$.
By the same arguments as in the proof
of \cite[Prop.~6.1]{Chinburg85} we obtain the following proposition.

\begin{lemma}\label{Serresequence}
We have the following exact sequence
\begin{equation}\label{exactsequenceSerre}
0\to N(p+1)\to\Nur(p+1)\xrightarrow{(F-1)\times 1}\Nur(p+1)\xrightarrow{\omega}\Z\to 0
\end{equation}
of $\ZG$-modules.
The extension class of this sequence is induced by the negative of the local fundamental class in $\mathrm{Ext}_\ZG^2(\Z,N^\times)$.
\end{lemma}

\begin{proof}
Analogous to the proof of \cite[Prop.~6.1]{Chinburg85}.
\end{proof}

Let
\[
\ff'=\Z[G]z_1\oplus\Z[G]z_2,
\]
\[
\ff_{\geq n}= \bigoplus_{j=n}^{p-1}\bigoplus_{k=1}^m\Z[G] v_{k,j}
\]
and let
\[\ff=\ff_{\geq 0}.\]

Note that the assignment $v_{k,j} \mapsto \alpha_kw_j$ induces an isomorphism
\begin{equation}\label{ff iso}
\ff_p := \ff \otimes_\Z \Zp \stackrel{\simeq}\lra \bigoplus_{j=0}^{p-1}\bigoplus_{k=1}^m\Z_p[G]\alpha_kw_j = \bigoplus_{j=0}^{p-1}\OKG w_j
\end{equation}
of free $\ZpG$-modules.

In the following we let $[x_1, \ldots, x_d]$ denote the class in $\Nur(p+1)$ represented by
$(x_1, \ldots, x_d)$. If $x = x_1 = \ldots = x_d$ we will often write $[x]$ instead of $[x, \ldots,x]$.

\begin{lemma}\label{basic diagram}
There is a commuting diagram
\[
\xymatrix{
0\ar[r]&X(2)\oplus \ff\ar[r]\ar[d]^{f_4}&\ff'\oplus \ff\ar[rr]^{\delta_2}\ar[d]^{f_3}&&\Z[G]z_0\ar[r]\ar[d]^{f_2}&\Z\ar[r]\ar[d]^{=}&0\\
0\ar[r]&N(p+1)\ar[r]&\Nur(p+1)\ar[rr]^{(F-1)\times 1}&&\Nur(p+1)\ar[r]&\Z\ar[r]&0.
}
\]
of $\ZG$-modules with
\[\begin{split}&\delta_2(z_1)=(b-1)z_0,\\&
\delta_2(z_2)=(a-1)z_0,\\&
\delta_2(v_{k,j})=0\text{ for all $k$ and $j$},\\&
f_2(z_0)=f_3(z_1)=[\theta_1,1,1,\dots,1],\\&
f_3(z_1)=[\theta_1,1,\dots,1],\\&
f_3(z_2)=[\gamma,\gamma,\dots,\gamma],\\&
f_3(v_{k,j})=1+\alpha_k(a-1)^j\theta \text{ for all $k$ and $j$}.\end{split}\]
Further, $X(2) := \mathrm{ker}(\delta_2\mid_{\ff'})$ and  $f_4$ is the restriction of $f_3$ to $X(2) \oplus \ff$.
\end{lemma}

\begin{proof}
Straightforward verification.
\end{proof}

The diagram in Lemma \ref{basic diagram} will be fundamental for our proof of Theorem \ref{main theorem}.
We will use the top exact sequence to compute groups of the form $\Ext_\ZG^2(\Z, \_)$. By Proposition \ref{rep of lfc} below we can then apply the 
recipe described in Section \ref{definition of EX} with $\Sigma = X(2) \oplus \ff$ and $\varphi = -f_4$ to compute $E(U_N^{(p+1)})_p$.

In the remainder of this subsection we will provide the proof of the following lemma.

\begin{lemma}\label{surjectivity lemma}
The homomorphism $f_4$ is surjective.
\end{lemma}

As a consequence we obtain

\begin{prop}\label{rep of lfc}
The map $-f_4$ represents the local fundamental class.
\end{prop}

\begin{proof}
This can be proved by mimicking the proof of \cite[Lemma 6.3]{Chinburg85}.
\end{proof}

\begin{lemma}\label{X2generators}
We have
\[X(2)=\langle(a-1)z_1-(b-1)z_2,T_bz_1,T_az_2\rangle_{\Z[G]}.\]
\end{lemma}

\begin{proof}
The inclusion ''$\supseteq$'' is immediate from the definition of $\delta_2$.
Let us consider the reverse inclusion. Let
\[
x=\sum_{i=0}^{p-1}\sum_{j=0}^{d-1}\alpha_{i,j}a^ib^jz_1+\sum_{i=0}^{p-1}\sum_{j=0}^{d-1}\beta_{i,j}a^ib^jz_2\in X(2)
\]
with $\alpha_{i,j}, \beta_{i,j} \in \Z$. From $\delta_2(x) = 0$ we derive
\begin{equation}\label{lemmaX2firsteq}
\alpha_{i,j-1}-\alpha_{i,j}+\beta_{i-1,j}-\beta_{i,j}=0
\end{equation}
for all $0\leq i<p$ and $0\leq j<d$. 
Here and in the following we regard all indices as integers modulo $p$ and $d$ respectively.

From (\ref{lemmaX2firsteq}) we deduce that
$\alpha := \sum_{i=0}^{p-1}\alpha_{i,j}$ does not depend on the choice of $j$. 
Now we are looking for integers $\gamma_{i,j},\mu_i,\nu_j$, for $0\leq i<p$ and $0\leq j<d$, such that
\[
\begin{split}x&=\sum_{i=0}^{p-1}\sum_{j=0}^{d-1}\gamma_{i,j}a^ib^j((a-1)z_1-(b-1)z_2)+\sum_{i=0}^{p-1}\mu_ia^iT_bz_1+\sum_{j=0}^{d-1}\nu_jb^jT_az_2\\
&=\sum_{i=0}^{p-1}\sum_{j=0}^{d-1}(\gamma_{i-1,j}-\gamma_{i,j}+\mu_i)a^ib^jz_1+\sum_{i=0}^{p-1}\sum_{j=0}^{d-1}(-\gamma_{i,j-1}+\gamma_{i,j}+\nu_j)a^ib^jz_2.
\end{split}
\]
So, in other words, the lemma is proved if we find integers $\gamma_{i,j},\mu_i,\nu_j$ such that
\begin{equation}\label{lemmaX2toprove1}
\alpha_{i,j}=\gamma_{i-1,j}-\gamma_{i,j}+\mu_i
\end{equation}
and
\begin{equation}\label{lemmaX2toprove2}
\beta_{i,j}=-\gamma_{i,j-1}+\gamma_{i,j}+\nu_j. 
\end{equation}
With  $\nu_j :=\beta_{0,j}$, $\mu_0:=\alpha$, $\mu_i=0$ for $i>0$ and  $\gamma_{i,j}=-\sum_{1\leq\ell\leq i}\alpha_{\ell,j}$ 
it is straightforward to verify  (\ref{lemmaX2toprove1}). Equality (\ref{lemmaX2toprove2}) is proved by an easy induction on $i$ using (\ref{lemmaX2firsteq}).
\end{proof}

We evaluate $f_4$ at the three special elements of $X(2)$ given by the last lemma.

\begin{lemma}\label{f4calc}
We have
\begin{equation}\label{f4one}
f_4((a-1)z_1-(b-1)z_2)=[\gamma]^{1-b}, 
\end{equation}
\begin{equation}\label{f4two}
f_4(T_bz_1)=[\theta_1]
\end{equation}
and
\begin{equation}\label{f4three}
f_4(T_az_2)=[\gamma]^{T_a}.
\end{equation}
\end{lemma}

\begin{proof} 
Straightforward computation using (\ref{galois action rules 1}), (\ref{galois action rules 2}), 
$\theta_1^b = \theta_1$ and $[\theta_1,1,\dots,1]^{a-1} =[\theta_1^b,1,\dots,1]^{a-1} = [\gamma^{b(F_0-1)},1,\dots,1]$.
\end{proof}

We write $\hat{f_4} \colon X(2)_p \oplus \ff_p \lra N(p+1)_p$ for the $p$-completion of $f_4$. For an element $\beta \in \OKG$
we write $\beta = \sum_{k=1}^m \lambda_k \alpha_k$ with uniquely determined $\lambda_k \in \ZpG$ and  according to (\ref{ff iso})
we set $\hat{f_4}(\beta w_j) := \prod_{k=1}^m f_4(v_{k,j})^{\lambda_k}$.

\begin{lemma}\label{lemma3Bley} 
Let $\beta\in \oo_{K}[G]$. Then we have for $j = 0, \ldots, p-1$
\[
\hat{f_4}(\beta w_j)\equiv 1+(a-1)^j\beta\theta \pmod{U_{N}^{(j+2)}}.\]
\end{lemma}

\begin{proof}
As above we write  $\beta = \sum_{k=1}^m \lambda_k \alpha_k$. We note that that for $n \ge 1$ the map $x \mapsto 1+x$ induces an 
isomorphism $\p_N^n / \p_N^{n+1} \simeq U_N^{(n)} / U_N^{(n+1)}$ of $\ZpG$-modules. 
By Lemma \ref{lemma1Bley} we have $v_N((a-1)^j \theta) = j+1$. Therefore,
\[\begin{split}
\hat{f_4}(\beta w_j)&=\prod_{k=1}^m(1+\alpha_k(a-1)^j\theta)^{\lambda_k}\\
&\equiv1+\sum_{k=1}^m\lambda_k\alpha_k(a-1)^j\theta\pmod{U_{N}^{(j+2)}}\\
&\equiv 1+\beta(a-1)^j\theta\pmod{U_{N}^{(j+2)}}.\end{split}\]
\end{proof}

\begin{lemma}\label{f4surj1}
For $j=0,\dots,p$, any element of $U_N^{(j+1)}/U_N^{(p+1)}$ is the image under $\hat{f_4}$ of an element in $(\ff_{\geq j})_p$.
\end{lemma}

\begin{proof}

For $j=p$, $\ff_{\geq p}=\{0\}$ and $U_N^{(p+1)}/U_N^{(p+1)}=\{0\}$, so the result is trivial.

We assume the result for $j+1$ and proceed by descending induction. If $x\in U_N^{(j+1)}/U_N^{(p+1)}$, 
then $x=1+\mu p\theta+\nu(a-1)^j\theta$ for some $\mu,\nu\in\oo_K[G]$ by Lemma \ref{lemma2Bley}.
Since $\mu p\theta\in\p_N^{p+1}$, Lemma \ref{lemma3Bley} implies 
\[
x\equiv 1+\nu(a-1)^j\theta\equiv \hat{f_4}(\nu w_j)\pmod{U_N^{j+2}}.
\]
This means that $x$ is the product of an element in the image of $(\ff_{\geq j})_p$ and an element in $U_N^{(j+2)}/U_N^{(p+1)}$, 
which is by assumption in the image of $(\ff_{\geq j+1})_p \subseteq (\ff_{\geq j})_p$. 
\end{proof}

After these preparations we are ready to provide the proof of Lemma \ref{surjectivity lemma}.

\begin{proof}[Proof of Lemma \ref{surjectivity lemma}.]
We recall the properties of $\gamma$ described in Lemma \ref{F0-1surjective}.
Since $a$ is in the inertia group, by (\ref{f4three}) we obtain
$f_4(T_az_2)=[\gamma]^{T_a}\equiv [\gamma]^p\pmod{\p_N}$.
Since $\gamma\equiv\zeta_{q^d-1}\pmod{\p_{N_0}}$, its class is a generator of $U_N/U_N^{(1)}$. Since $p$ is co-prime to the order $q^{d}-1$ of $U_N/U_N^{(1)}$, we conclude that the projection of $f_4(X(2))$ onto $N(1)$ contains the torsion subgroup $U_N/U_N^{(1)}$ of $N(1)$.

By (\ref{f4two}) we obtain
$f_4(T_bz_1)=[\theta_1]$ and recall that $\theta_1$ is a prime element in $N$.
We conclude that any element of $N(p+1)$ is the product of an element in the image of $f_4$ and an element in $U_N^{(1)}/U_N^{(p+1)}$.
It therefore remains to prove that $U_N^{(1)}/U_N^{(p+1)}$ is also in the image of $f_4$. Since $U_N^{(1)}/U_N^{(p+1)}$ is a finite $p$-group
this follows immediately from Lemma \ref{f4surj1}.
\end{proof}

\end{subsection}

\begin{subsection}{The kernel of $\hat{f_4}$}\label{kernel of f_4}
In order to compute a representative of $E(\exp(\calL))_p$ we have to compute a $\ZpG$-basis of
$\ker(f_4)_p = \ker(\hat{f_4})$. As a first step in this direction we construct certain explicit elements in $\ker(\hat{f_4})$
and show that they form a complete set of generators. We then manipulate this set of generators in order to obtain a basis. The main result
is summarized in Proposition \ref{pm+1generators}.

\begin{lemma}\label{deft1}
Let $\tilde m$ be an integer such that $m\tilde m \equiv 1 \pmod d$.
Set 
\begin{eqnarray*}\tilde t_1 &:=& (a-1)z_1-(b-1)z_2+
\left(\sum_{i=2}^{m}\alpha_{i}b^{1-(i-2)\tilde m}+\left(\alpha_1-\sum_{i=2}^{m}\alpha_{i}\right)b^{\tilde m}\right)w_0.
\end{eqnarray*}
 Then there exists $y_1\in (\ff_{\geq 1})_p$, such that $t_1 := \tilde t_1 + y_1 \in \ker(\hat{f_4})$.
\end{lemma}

\begin{proof}
We calculate $\hat{f_4}(\tilde t_1)$ modulo $\p_N^2$. First we recall that $\gamma\equiv\zeta_{q^d-1}(1+x_2\theta_1)\pmod{\p_{N_0}^2}$ 
by Lemma \ref{F0-1surjective}. Note also that for any integer $s \ge 2$ one has $\zeta_{q^d - 1} = 1$ in the $p$-completion $N(s)_p$ of $N(s)$.
By the definition of $\hat f_4$, Lemma \ref{f4calc} and Lemma \ref{lemma3Bley},
\[
\begin{split}
\hat{f_4}(\tilde t_1)
&\equiv\gamma^{1-b}(1+\alpha_1{b^{\tilde m}}\theta)\prod_{i=2}^m\left((1+\alpha_{i}b^{1-(i-2)\tilde m}\theta)(1+\alpha_ib^{\tilde m}\theta)^{-1}\right)\pmod{\p_{N_0}^2}\\
&\equiv (1+x_2\theta)^{1-b}\left(1+\alpha_1\theta^{b^{\tilde m}}+\sum_{i=0}^{m-2}\alpha_{i+2}\theta^{b^{1-i\tilde m}}-
\sum_{i=0}^{m-2}\alpha_{i+2}\theta^{b^{\tilde m}}\right)\pmod{\p_{N_0}^2}.
\end{split}\]
Now we see that
\begin{equation}\label{zz1}
\begin{split}
(1+x_2\theta_1)^{1-b}&\equiv(1+x_2\theta_1)(1+x_2\theta_1)^{-b}\equiv(1+x_2\theta_1)(1-x_2^b\theta_1)\pmod{\p_{N_0}^2}\\
&\equiv 1+x_2\theta_1-x_2^b\theta_1\equiv 1+\left(\left(\frac{x_2}{\alpha_1}\right)^F-\frac{x_2}{\alpha_1}\right)^b\alpha_1\theta_1\pmod{\p_{N_0}^2}\\
&\equiv 1+\left(\left(\frac{x_2}{\alpha_1}\right)^q-\frac{x_2}{\alpha_1}\right)^b\alpha_1\theta_1\pmod{\p_{N_0}^2}.
\end{split}
\end{equation}
By the choice of the basis $\alpha_1, \ldots, \alpha_m$ in (\ref{choice of alpha}) we have 
$A^{f^i}=\frac{\alpha_{i+2}}{\alpha_1}$ for $i=0,\dots, m-2$ and $A^{f^{m-1}}=1-\sum_{i=0}^{m-2}A^{f^i}=1-\sum_{i=0}^{m-2}\frac{\alpha_{i+2}}{\alpha_1}$. So we get
\[\begin{split}\left(\frac{x_2}{\alpha_1}\right)^q-\frac{x_2}{\alpha_1}&\equiv\sum_{i=0}^{m-1}\left(\left(\frac{x_2}{\alpha_1}\right)^{p^{i+1}}-\left(\frac{x_2}{\alpha_1}\right)^{p^i}\right)\pmod{\p_{N_0}}\\
&\equiv\sum_{i=0}^{m-1}\left(\left(\frac{x_2}{\alpha_1}\right)^{p}-\left(\frac{x_2}{\alpha_1}\right)\right)^{p^i}\pmod{\p_{N_0}}\\
&\stackrel{(i)}\equiv -\sum_{i=0}^{m-1}\left(A\theta_2\right)^{p^i}\pmod{\p_{N_0}}\\
&\stackrel{(ii)}\equiv -\sum_{i=0}^{m-2} A^{p^i}\theta_2^{p^{im\tilde m}}-A^{p^{m-1}}\theta_2^{p^{(m-1)m\tilde m}}\pmod{\p_{N_0}}\\
&\equiv -\sum_{i=0}^{m-2} \frac{\alpha_{i+2}}{\alpha_1}\theta_2^{q^{i\tilde m}}-\left(1-\sum_{i=0}^{m-2}\frac{\alpha_{i+2}}{\alpha_1}\right)\theta_2^{q^{1-\tilde m}}\pmod{\p_{N_0}}.
\end{split}\]
The congruence (i) follows from our choice of $x_2$ and (ii) is immediate from $\theta_2 \in \tilde{K'}$.

Combining the last congruence with the computation in (\ref{zz1}) and recalling that $\theta_1^b=\theta_1$ we obtain
\[\begin{split}(1+x_2\theta_1)^{1-b}&\equiv 1-\left(\sum_{i=0}^{m-2} \frac{\alpha_{i+2}}{\alpha_1}\theta_2^{q^{i\tilde m}}+\left(1-\sum_{i=0}^{m-2}\frac{\alpha_{i+2}}{\alpha_1}\right)\theta_2^{q^{1-\tilde m}}\right)^b\alpha_1\theta_1\pmod{\p_{N_0}^2}\\
&\equiv 1-\sum_{i=0}^{m-2} \alpha_{i+2}\theta^{b^{1-i\tilde m}}-\alpha_1\theta^{b^{\tilde m}}+\sum_{i=0}^{m-2}\alpha_{i+2}\theta^{b^{\tilde m}}\pmod{\p_{N_0}^2}.
\end{split}\]
So we conclude that
\[\hat{f_4}(\tilde t_1)\equiv 1-\left(\alpha_1\theta^{b^{\tilde m}}+\sum_{i=0}^{m-2}\alpha_{i+2}\theta^{b^{1-i\tilde m}}-\sum_{i=0}^{m-2}\alpha_{i+2}\theta^{b^{\tilde m}}\right)^2\equiv 1\pmod{\p_{N_0}^2}.\]
Therefore $\hat{f_4}(\tilde t_1)^{-1}\in U_N^{(2)}/U_N^{(p+1)}$ and by Lemma \ref{f4surj1} there exists $y_1\in (\ff_{\geq 1})_p$ 
such that $\hat{f_4}(y_1)\equiv \hat f_4(\tilde t_1)^{-1}\pmod{\p_N^{p+1}}$, i.e. $\tilde t_1+y_1\in\ker (\hat{f_4})$.
\end{proof}

\begin{lemma}\label{deft2}
The element
\[
t_2:=T_az_2-\beta w_{p-1} \text{ with } \beta = \begin{cases} \alpha_1, & \text{if } m = 1, \\
                                                              \alpha_2, & \text{if } m > 1,
                                                \end{cases}
\]
is in the kernel of $\hat{f_4}$.
\end{lemma}

\begin{proof}
Since $\zeta_{q^d - 1} = 1$ in $N(p+1)_p$ and $\gamma \equiv \zeta_{q^d - 1} (1 + x_2\theta_1) \pmod{\p_{N_0}^2}$ the formulae in Lemma \ref{f4calc} and Lemma \ref{lemma3Bley} imply
\[
\begin{split}
\hat{f_4}(t_2)&\equiv\gamma^{T_a}(1-\beta(a-1)^{p-1}\theta)\pmod{\p_{N_0}^{p+1}}\\
&\equiv\mathcal N_{N_0 / K_{\mathrm{nr}}}(1+x_2\theta_1)(1-\beta(a-1)^{p-1}\theta)\pmod{\p_{N_0}^{p+1}}.
\end{split}
\]
Note that by \cite[Sec.~V.6, Prop.~8]{SerreLocalFields} we know that
$\mathcal N_{N_0/\Kur}U_{N_0}^{(2)}\subseteq U_{\Kur}^{(2)}\subseteq U_{N_0}^{(2p)}\subseteq U_{N_0}^{(p+1)}$ so that it suffices to work with $\gamma$ modulo
$\p_{N_0}^2$.
Using Lemma \ref{norm_power_a-1}, \cite[Sec.~V.3, Lemma 5]{SerreLocalFields},  the fact that 
$x_2\in \oo_{\Kur}$ and Lemma \ref{normtrace}, we get
\[\begin{split}\hat{f_4}(t_2)&\equiv\mathcal N_{N_0 / K_{\mathrm{nr}}}(1+x_2\theta_1)(1-\beta\theta_2p)\pmod{\p_{N_0}^{p+1}}\\
&\equiv (1+\Tr_{N_0/\Kur}(x_2\theta_1)+\mathcal N_{N_0/\Kur}(x_2\theta_1))(1-\beta\theta_2p)\pmod{\p_{N_0}^{p+1}}\\
&\equiv (1+x_2p-x_2^p \alpha_1^{1-p}p)(1-\beta\theta_2p)\pmod{\p_{N_0}^{p+1}}\\
&\equiv 1+(x_2-x_2^p \alpha_1^{1-p}-\beta\theta_2)p\pmod{\p_{N_0}^{p+1}}.\end{split}\]
By the choice of $x_2$ made after Lemma \ref{xp-x+Atheta2} we have
\[x_2^p\alpha_1^{1-p}=\alpha_1\left(\frac{x_2}{\alpha_1}\right)^p\equiv\alpha_1\cdot\left(\frac{x_2}{\alpha_1}-A\theta_2\right)\equiv x_2-A\alpha_1\theta_2\equiv x_2-\beta\theta_2\pmod{\p_{\Kur}},\]
so that $\hat{f_4}(t_2) \equiv 1 \pmod{\p_{N_0}^{p+1}}$.
\end{proof}

\begin{lemma}\label{genkerf4X2}
The elements $t_1$ and $t_2$ generate $\left( \ker \hat{f_4} + \ff_p \right)/\ff_p$ as a $\Z_p[G]$-module.
\end{lemma}

\begin{proof}
We write $W \sseq X(2)_p \oplus \ff_p$ for the $\ZpG$-submodule which is generated by $\ff_p$, $t_1$ and $t_2$.
For each $x\in\ker \hat{f_4}$ we have to show that $x \in W$. In the following all congruences are modulo $W$.
By Lemma \ref{X2generators} there exist $x_1,x_2,x_3\in\Z_p[G]$ such that 
\[
x \equiv  x_1((a-1)z_1-(b-1)z_2)+x_2T_bz_1+x_3T_az_2.
\]
From the definitions of $t_1$ and $t_2$ we immediately obtain
$$T_az_2 \in W \text{ and } (a-1)z_1 - (b-1)z_2 \in W.$$
Hence $x \equiv x_2T_bz_1$. Without loss of generality we may assume $x_2 \in \Z_p[a]$.
By considering $\hat{f_4}(x)$ modulo $U_N^{(1)}$ and using Lemma \ref{f4calc}, 
we see that, to kill the $[\theta_1]$, $x_2$ must be in the augmentation ideal.
Therefore there exists $x_4 \in \Z_p[a]$ such that   $x_2 = x_4 (a-1)$. Then
$x \equiv x_4 (a-1) T_bz_1 = x_4T_b \left( (a-1)z_1 - (b-1)z_2 \right) \equiv 0$.
\end{proof}

\begin{lemma}
Let $0\leq j\leq p-1$, $1\leq k\leq m$. Then there exists $\mu_{j,k}\in (\ff_{\geq j+2})_p$ such that the element
\[s_{j,k}=\alpha_k(a-1) w_j-\alpha_kw_{j+1}+\mu_{j,k}\]
is in the kernel of $\hat{f_4}$. Here $w_p$ should be interpreted as $0$.
\end{lemma}

\begin{proof}
For $l\geq 1$ we put $\eta_l := \alpha_k (a-1)^{l-1} \theta$. Note that $v_N(\eta_l) = l$ for $1 \le l \le p$.
In the following all congruences are modulo ${U_N^{(j+3)}}$.
Then, for $0\leq j<p-1$, we compute
\[\begin{split}
&\hat{f_4}((a-1) \alpha_kw_j-\alpha_kw_{j+1})\\
&\qquad=(1+\eta_{j+1})^{a-1} (1 + \eta_{j+2})^{-1}\\
&\qquad\equiv (1+a\eta_{j+1}) (1 - \eta_{j+1} + \eta_{j+1}^2) (1 - \eta_{j+2})\\
&\qquad\equiv (1-\eta_{j+1} +\eta_{j+1}^2 +a\eta_{j+1} -(a\eta_{j+1})\eta_{j+1}) (1 - \eta_{j+2})\\
&\qquad\equiv (1+\eta_{j+2} -\eta_{j+1}\eta_{j+2}) (1 - \eta_{j+2})\\
&\qquad\equiv (1+\eta_{j+2}) (1 - \eta_{j+2}) \equiv 1.\\
\end{split}\]
For $j=p-1$, using Lemma \ref{norm_power_a-1}, we have
\[\hat f_4((a-1)\alpha_kw_{p-1})=(1+\alpha_k(a-1)^{p-1}\theta)^{a-1}\equiv (1+\alpha_k\theta_2p)^{a-1}=1\pmod{U_N^{(p+1)}}.\]
Now we conclude using Lemma \ref{f4surj1} as in the proof of Lemma \ref{deft1}.
\end{proof}

By construction, any element of $\ff'_p\oplus \ff_p$ can be written as a linear 
combination of $z_1$, $z_2$ and $\alpha_iw_j$, for $i=1,\dots,m$ and $j=0,\dots,p-1$ with coefficients in $\Z_p[G]$. 
In this context we will speak of $z_1$-, $z_2$- and $\alpha_iw_j$-components of elements of $\ker \hat{f_4}$.

We recall that $\tilde m$ is an integer such that $m\tilde m \equiv 1 \pmod d$.

\begin{lemma}
The element
\[
r_1=T_a t_1 + (b-1)t_2
\]
belongs to $\ker \hat{f_4}\cap \ff_p$ and its $\alpha_1w_0$-component is $b^{\tilde m} T_a$.
\end{lemma}

\begin{proof}
By the definition of  $t_1$ and $t_2$ in  Lemmata \ref{deft1} and \ref{deft2} the element $r_1$ belongs to $\ker(\hat f_4)$. Hence it suffices to prove that 
the $z_1$- and $z_2$-components of $r_1$ are zero and the $\alpha_1w_0$-component is $b^{\tilde m} T_a$.
This follows by a straightforward computation.

\end{proof}

\begin{lemma}
The elements
\[r_k=\alpha_k T_a w_0+(b^{-\tilde m}\alpha_{k+1}-\alpha_k)w_{p-1},\]
for $1<k<m$, and
\[r_m=\alpha_m T_a w_0+\left(b^{-\tilde m}\alpha_1-b^{-\tilde m}\sum_{i=2}^m\alpha_i-\alpha_m\right)w_{p-1}\]
are in the kernel of $\hat{f_4}$.
\end{lemma}

\begin{proof}
For $1<k\leq m$, using \cite[Sec.~V.3, Lemma 4 and Lemma 5]{SerreLocalFields}, Lemma \ref{defntheta1} and Lemma \ref{normtrace},
\[\begin{split}\hat{f_4}(\alpha_kT_aw_0)&\equiv\mathcal N_{N/K'}(1+\alpha_k\theta) \pmod{\p_N^{p+1}}\\
&\equiv 1+\mathcal N_{N/K'}(\alpha_k\theta)+\Tr_{N/K'}(\alpha_k\theta)\pmod{\p_N^{p+1}}\\
&\equiv 1+(\theta_2\alpha_k)^p\mathcal N_{N/K'}(\theta_1)+\theta_2\alpha_k\Tr_{N/K'}(\theta_1)\pmod{\p_N^{p+1}}\\
&\equiv 1-(\theta_2\alpha_k)^p\alpha_1^{1-p}p+\theta_2\alpha_kp\pmod{\p_N^{p+1}}\\
&\equiv 1+\left(\frac{\theta_2\alpha_k}{\alpha_1}-\left(\frac{\theta_2\alpha_k}{\alpha_1}\right)^p\right)\alpha_1p\pmod{\p_N^{p+1}}
.\end{split}\]
Now we note that $\theta_2^p\equiv\theta_2^{p^{m\tilde m}}\equiv\theta_2^{q^{\tilde m}}\equiv \theta_2^{b^{-\tilde m}} \pmod{\p_{\tilde K'}}$
and by (\ref{choice of alpha}) we have for $1<k<m$,
\[
\left(\frac{\alpha_k}{\alpha_1}\right)^p=\left(A^{f^{k-2}}\right)^p \equiv A^{p^{k-1}} \equiv \frac{\alpha_{k+1}}{\alpha_1} \pmod{\p_K}
\]
and
\[
\left(\frac{\alpha_m}{\alpha_1}\right)^p=\left(A^{f^{m-2}}\right)^p \equiv A^{f^{m-1}} = 1-\sum_{i=0}^{m-2}A^{f^i} = 1-\sum_{i=0}^{m-2}\frac{\alpha_{i+2}}{\alpha_1}
\pmod{\p_K}.
\]
Therefore, for $1<k<m$,
\[
\begin{split}\hat{f_4}(\alpha_kT_aw_0)&\equiv 1+\left(\frac{\theta_2\alpha_k}{\alpha_1}-\frac{\theta_2^{b^{-\tilde m}}\alpha_{k+1}}{\alpha_1}\right)\alpha_1p\pmod{\p_N^{p+1}}\\
&\equiv 1+\left(\alpha_k\theta_2-\alpha_{k+1}\theta_2^{b^{-\tilde m}}\right)p\pmod{\p_N^{p+1}}
\end{split}
\]
and
\[\begin{split}\hat{f_4}(\alpha_mT_aw_0)&\equiv 1+\left(\frac{\theta_2\alpha_m}{\alpha_1}-\theta_2^{b^{-\tilde m}}\left(1-\sum_{i=0}^{m-2}\frac{\alpha_{i+2}}{\alpha_1}\right)\right)\alpha_1p\pmod{\p_N^{p+1}}\\
&\equiv 1+\left(\alpha_m\theta_2-\alpha_1\theta_2^{b^{-\tilde m}}+\sum_{i=2}^{m}\alpha_{i}\theta_2^{b^{-\tilde m}}\right)p\pmod{\p_N^{p+1}}.\end{split}\]
Recalling Lemma \ref{norm_power_a-1}, for $1<k<m$ we obtain
\[\begin{split}\hat{f_4}((b^{-\tilde m}\alpha_{k+1}-\alpha_k)w_{p-1})&\equiv(1+\alpha_{k+1}p\theta_2)^{b^{-\tilde m}}(1+\alpha_kp\theta_2)^{-1}\pmod{\p_{N}^{p+1}}\\
&\equiv 1-\left(\alpha_k\theta_2-\alpha_{k+1}\theta_2^{b^{-\tilde m}}\right)p \pmod{\p_{N}^{p+1}}
\end{split}\]
and
\[\begin{split}&\hat{f_4}\left(\left(b^{-\tilde m}\alpha_1-b^{-\tilde m}\sum_{i=2}^m\alpha_i-\alpha_m\right)w_{p-1}\right)\\
&\qquad\equiv(1+\alpha_{1}p\theta_2)^{b^{-\tilde m}}\prod_{i=2}^m(1+\alpha_ip\theta_2)^{-b^{-\tilde m}}(1+\alpha_mp\theta_2)^{-1}\pmod{\p_{N}^{p+1}}\\
&\qquad\equiv 1-\left(\alpha_m\theta_2-\alpha_{1}\theta_2^{b^{-\tilde m}}+\sum_{i=2}^m\alpha_i\theta_2^{b^{-\tilde m}}\right)p \pmod{\p_{N}^{p+1}}
\end{split}\]
Therefore we conclude that in all cases $r_k\in\ker \hat{f_4}$.
\end{proof}

\begin{lemma}\label{genkerh}
The $pm+m$ elements $r_{k},s_{j,k}$ for $0\leq j\leq p-1$, $1\leq k\leq m$ generate $\ker \hat{f_4}\cap\ff_p$ 
as a $\Z_p[G]$-module.
\end{lemma}

\begin{proof}
We define elements $r_{j, k}$, for $0\leq j\leq p-1$, $1\leq k\leq m$, as follows: $r_{0, k}=r_k$ and $r_{j,k}=T_as_{j-1,k}$, for $j>0$. It will suffice to show that the $2pm$ elements $r_{j,k},s_{j,k}$ for $0\leq j\leq p-1$, $1\leq k\leq m$ are generators of 
$\ker\hat{f}_4 \cap \ff_p$.

It is obvious that they generate $\ker \hat{f_4}\cap (\ff_{\geq p})_p=\{0\}$. Let us assume they generate $\ker \hat{f_4}\cap (\ff_{\geq j+1})_p$, 
for some $j<p$, and let us prove that they generate $\ker \hat{f_4}\cap (\ff_{\geq j})_p$. Let $x\in \ker \hat{f_4}\cap (\ff_{\geq j})_p$. 
We can write $x=\lambda_1 w_j+\lambda_2$, for some $\lambda_1\in \oo_{K}[G]$ and $\lambda_2\in (\ff_{\geq j+1})_p$. Then by Lemma \ref{lemma3Bley} we have
$\hat{f_4}(x)\equiv 1+(a-1)^j\lambda_1\theta\pmod{U_{N}^{(j+2)}},$
which must be congruent to $1$ by the assumption that $x\in\ker \hat{f_4}$. Hence
\begin{equation}(a-1)^j\lambda_1\theta\in\p_N^{j+2}.\label{inlemmakerf4}\end{equation}
By Lemma \ref{lemma1Bley}, if $v_N\left(\lambda_1\theta\right) = 1$, then $v_N\left((a-1)^j\lambda_1\theta\right) = j+1$ (recall that $j<p$), 
and this contradicts (\ref{inlemmakerf4}). Hence $v_N\left(\lambda_1\theta\right)>1$, so that Lemma \ref{lemma2Bley} implies
$\lambda_1\theta\in\p_{N}^2=(T_a,a-1)\theta$. It follows that $\lambda_1 \in (T_a, a-1)$.
So in particular $\lambda_1w_j$ is a sum of a linear combination of the elements $r_{j,k}$ and $s_{j,k}$, for $k=1,\dots,m$, and an element in $(\ff_{\geq j+1})_p$. Hence also $x$ is a combination of the elements $r_{j,k}$ and $s_{j,k}$ and an element in $(\ff_{\geq j+1})_p$, which must also be in $\ker \hat{f_4}$. To conclude we only need to recall the inductive hypothesis.
\end{proof}

Now we consider the $\alpha_iw_{p-1}$-components, $i = 1,\ldots,m$, of $t_2$ and $r_k$ for $k=2,\dots,m$. 
We write these components as the columns of an $m\times m$ matrix $\M$. We have to distinguish two cases. If $m>1$,

\[
\M=\left(\begin{matrix}
0&0&0&\dots&0&0&0&b^{-\tilde m}\\
-1&-1&0&\dots&0&0&0&-b^{-\tilde m}\\
0&b^{-\tilde m}&-1&\dots&0&0&0&-b^{-\tilde m}\\
0&0&b^{-\tilde m}&\dots&0&0&0&-b^{-\tilde m}\\
\vdots&\vdots&\vdots&\ddots&\vdots&\vdots&\vdots&\\
0&0&0&\dots&b^{-\tilde m}&-1&0&-b^{-\tilde m}\\
0&0&0&\dots&0&b^{-\tilde m}&-1&-b^{-\tilde m}\\
0&0&0&\dots&0&0&b^{-\tilde m}&-1-b^{-\tilde m}
\end{matrix}\right).
\]

If $m=1$, then the matrix is determined only by $t_2$ and, recalling its definition from Lemma \ref{deft2}, we get
$\M=(-1)$.

\begin{lemma}\label{detM}
The determinant of $\M$ is $(-1)^{m}b^{-\tilde m(m-1)}=(-1)^{m}b^{\tilde m-1}$.
\end{lemma}

\begin{proof}
This is an easy calculation.
\end{proof}

\begin{lemma}\label{sp-1}
For $k=1,\dots,m$ we have $s_{p-1,k} \in \langle t_2, r_2, \ldots, r_m \rangle_{\ZpG}$.
\end{lemma}

\begin{proof}
We fix $k$ such that $1 \le k \le m$. Recall that $s_{p-1, k} = \alpha_k (a-1)w_{p-1}$.
By Lemma \ref{detM} there is a $\ZpG$-linear combination $x$ of $t_2$ and the elements $r_i$, $i=2,\dots,m$, 
such that the $\alpha_kw_{p-1}$-component of $x$ is $1$ and the $\alpha_jw_{p-1}$-components for $j\neq k$ are zero. 
The components of $x$ outside of $\alpha_iw_{p-1}$ are (by the definition of $t_2$ and $r_i$) always multiples of $T_a$. Therefore we can conclude that $(a-1)x=s_{p-1,k}$. 
\end{proof}

We are now ready to state and prove the main result of this subsection.

\begin{prop}\label{pm+1generators}
The $pm+1$ elements $t_1,t_2$, $r_k$, for $k=2,\dots,m$, $s_{j,k}$ for $0\leq j\leq p-2$, $1\leq k\leq m$ 
constitute a $\ZpG$-basis of  $\ker \hat{f_4}$.
\end{prop}
\begin{proof}
By \cite[Lemma 3.7]{BlBu03}, $\ker\hat f_4$ is $\Z_p[G]$-free. From (\ref{theta tilde}) with $r=pm+2$ we deduce that the $\Z_p[G]$-rank is $pm+1$. It therefore suffices to show that $t_1,t_2$, $r_k$, for $k=2,\dots,m$, $s_{j,k}$ for $0\leq j\leq p-2$, $1\leq k\leq m$ generate $\ker \hat f_4$.

By Lemma \ref{sp-1} and the definition of $r_1$, it is enough to show that the $pm+m+2$ elements $t_1,t_2$ and $r_k,s_{j,k}$ for $0\leq j\leq p-1$, $1\leq k\leq m$ are generators. This has been shown in Lemma \ref{genkerf4X2} and Lemma \ref{genkerh}.
\end{proof}

Writing the $z_1$-, $z_2$- and $\alpha_iw_j$-components of the above generators as the columns of a matrix $\mm$ we obtain  
\[\mm=\left(\begin{matrix}
a-1&0&0&0&0&\cdots&0&0\\
 1-b&T_a&0&0&0&\cdots&0&0\\
v&0&T_a\tilde I&(a-1)I&0&\cdots&0&0\\
*&0&0&-I&(a-1)I&\cdots&0&0\\
*&0&0&*&-I&\cdots&0&0\\
\vdots&\vdots&\vdots&\vdots&\vdots&\ddots&\vdots&\vdots\\
*&0&0&*&*&\cdots&-I&(a-1)I\\
*&-e_2&\tilde \M&*&*&\cdots&*&-I
\end{matrix}\right),
\]
where $I$ is the $m \times m$-identity matrix and $\tilde \M$ and $\tilde I$ denote the matrices obtained by removing the first column of $\M$ and $I$ respectively. If $m > 1$ the vector $e_2$ is the second vector in the canonical basis 
of $\ZpG^m$, and if $m=1$ we set $e_2 = (1)$. Finally, $v \in \ZpG^m$ is a vector with first component $b^{\tilde m}$.
\end{subsection}

\begin{subsection}{Computation of the representative of $E(\exp(\calL))_p$}

We recall that $G = \gal(N/K) = \langle a \rangle \times \langle b \rangle$. Any irreducible character $\psi$ of $G$
decomposes as $\psi = \chi\phi$, where $\chi$ is an irreducible character of $\langle a \rangle$ and  
$\phi$ an irreducible character of $\langle b \rangle$.

We also recall that we always identify $K_0(\ZpG, \Qp)$ with $\QpG^\times / \ZpG^\times$. The following proposition
describes a representative of  $E(\exp(\calL))_p$ in $\QpG^\times$, which we regard as a subset of 
$\Q_p^c[G]^\times\simeq \bigoplus_{\chi, \phi} \Q_p^{c, \times}$.

\begin{prop}\label{representative prop}
We assume the setting introduced in Section \ref{section setting 1}. For $\calL = \p_N^{p+1}$ the element $E(\exp(\mathcal L))_p$ 
is represented by $\epsilon \in \QpG^\times$ where 
\[\epsilon_{\chi\phi}=\begin{cases}
dp^m&\text{if $\chi=\chi_0$ and $\phi=\phi_0$,}\\
\frac{\phi(b)^{\tilde m}}{1-\phi(b)}p^m&\text{if $\chi=\chi_0$ and $\phi\neq\phi_0$,}\\
(-1)^{m+1}\phi(b)^{\tilde m-1}(\chi(a)-1)^{m(p-1)}&\text{if $\chi\neq\chi_0$}.
\end{cases}
\] 
\end{prop}

\begin{proof}
By Proposition \ref{rep of lfc} the  map $-f_4$ represents the local fundamental class. 
Following the recipe described in Section \ref{definition of EX} we therefore consider the following diagram.
\[\xymatrix{
&0\ar[d]&0\ar[dr]&&0\\&\ker f_4\ar[d]^{i_1}&&\ker \delta_1\ar@/_/@{.>}[dl]_{\sigma}\ar[dr]^{i_3}\ar[ur]\\
0\ar[r]&X(2)\oplus \ff\ar[r]^{i_2}\ar@/_/[d]_{-f_4}&\ff'\oplus \ff\ar@/_/[ur]_{\delta_2}\ar[rr]^{\delta_2}&&\Z[G]z_0\ar@/_/[r]_{\quad \delta_1}&\Z\ar@/_/@{.>}[l]_{\tau}\ar[r]&0\\
&N(p+1)\ar[d]\ar@/_/@{.>}[u]_{\rho} & &\\
& 0
}
\]
Here the dotted maps $\tau,\sigma,\rho$ denote $G$-equivariant 
splitting morphisms. They only exist after tensoring with $\Q$. Since we are only
interested in the $p$-part of $E(\exp(\mathcal L))$ we tensor right away with $\Qp$. In the following, if $Y$
is a $\Z$-module, we set $Y_\Qp := \Qp \otimes_\Z Y$.

Explicitly, we define $\tau:\Q_p\to\Q_p[G]z_0$ by setting $\tau(1)=e_Gz_0$. For the definition of $\sigma$ we first note that
$(\ker\delta_1)_\Qp$ is generated by $(1-e_G)z_0$ as a $\QpG$-module. It is easy to see that
$\sigma((1-e_G)z_0):=\frac{1-e_b}{b-1}e_az_1+\frac{1-e_a}{a-1}z_2$ determines a well defined splitting 
$\sigma: (\ker \delta_1)_\Qp \to (\ff'\oplus \ff)_\Qp$. Here $\frac{1-e_b}{b-1}$ denotes the
inverse of $b-1$ on the $(1 -e_b)$-component of $\Q_p[b]$. Analogously we define $\frac{1-e_a}{a-1}$.

Finally, we also need to define $\rho:N(p+1)_{\Q_p}\to (X(2) \oplus \ff)_{\Q_p}$. 
We have $N(p+1)_{\Q_p}=\langle\theta_1\rangle_{\Q_p} \simeq \Qp$ because $U_N / U_N^{(p+1)}$ is torsion. 
We  define $\rho(\theta_1) :=-e_aT_bz_1$ and easily see that it defines a $G$-equivariant splitting of $-f_4$.

In the following all maps are morphisms of $\QpG$-modules, even though this will not be apparent in the notation.
The isomorphism $\tilde\theta$ of Section \ref{definition of EX} specialized to our situation is now explicitly given by
\[\begin{split}
\tilde\theta:(\ker f_4)_{\Q_p}\oplus\QpG&\xrightarrow{(\mathrm{id},(\tau,i_3)^{-1})}(\ker f_4)_{\Q_p}\oplus(\Qp\oplus (\ker d_1)_{\Q_p})\\
&\xrightarrow{(\mathrm{id},\nu_N^{-1},\mathrm{id})}(\ker f_4)_{\Q_p}\oplus N(p+1)_{\Q_p}\oplus (\ker d_1)_{\Q_p}\\
&\xrightarrow{(i_1,\rho,\mathrm{id})}(X(2)\oplus \ff)_{\Q_p}\oplus (\ker d_1)_{\Q_p}\\
&\xrightarrow{(i_2,\sigma)}\ff'_{\Q_p}\oplus \ff_{\Q_p}.
\end{split}\]
We fix $\ZpG$-bases of $\ker\hat{f_4} \oplus \ZpG$ and $\ff'_p \oplus \ff_p$, respectively. For $\ker\hat{f_4} \oplus \ZpG$ we take
$(0,1), (v_l, 0)$, $1 \le l \le pm+1$, where $v_l$ runs through the elements specified in Proposition \ref{pm+1generators}. For  $\ff'_p \oplus \ff_p$
we simply use the basis $z_1, z_2, \alpha_kw_j$, $1 \le k \le m, 0 \le j \le p-1$. 
We now compute the matrix $A_{\tilde\theta}$ with respect to these bases.
Following the definition of $\tilde\theta$ we get
\[\begin{split}
\tilde\theta:(0,1)&\mapsto (0,1,1-e_G)\\
&\mapsto(0,\theta_1,1-e_G)\\
&\mapsto(-e_aT_bz_1,1-e_G)\\
&\mapsto -e_aT_bz_1+\frac{1-e_b}{b-1}e_az_1+\frac{1-e_a}{a-1}z_2.
\end{split}\]
Writing the $z_1$-, $z_2$- and $\alpha_iw_j$-components of $\tilde\theta((0,1))$ as a column vector we obtain
\[
w=\left(\begin{matrix}
e_a \left( \frac{1-e_b}{b-1} - T_b \right) \\
\frac{1-e_a}{a-1}\\
0\\
0\\
\vdots\\
0
\end{matrix}\right).
\]
Since $\tilde\theta |_{\ker\hat{f_4}} = i_2 \circ i_1$ is the inclusion, we obtain $A_{\tilde\theta} = (w, \mm)$,
which is the matrix whose columns are $w$ and the columns of the matrix $\mm$ defined at the end of Section \ref{kernel of f_4}.

\noindent
{\bf Case 1:} $\chi=1$ and $\phi=1$.\\
Here $ (\chi\phi)( A_{\tilde\theta})$ is of the form
\[\left(\begin{matrix}
-d&0&0&0&0&0&\cdots&0&0\\
0&0&p&0&0&0&\cdots&0&0\\
0&\chi\phi(v)&0&p\tilde I&0&0&\cdots&0&0\\
0&*&0&0&-I&0&\cdots&0&0\\
0&*&0&0&*&-I&\cdots&0&0\\
\vdots&\vdots&\vdots&\vdots&\vdots&\vdots&\ddots&\vdots&\vdots\\
0&*&0&0&*&*&\cdots&-I&0\\
0&*&-e_2&\chi\phi(\tilde\M)&*&*&\cdots&*&-I
\end{matrix}\right),\]
where we recall that the first component of the vector $v$ is $b^{\tilde m}$.
The determinant is
$(-1)^{(p-1)m}dp^m=dp^m$.

\medskip

\noindent
{\bf Case 2:} $\chi=1$ and $\phi\neq 1$. \\
In this case $ (\chi\phi)( A_{\tilde\theta} )$ is of the form
\[\left(\begin{matrix}
\frac{1}{\phi(b)-1}&0&0&0&0&0&\cdots&0&0\\
0& 1-\phi(b)&p&0&0&0&\cdots&0&0\\
0&\chi\phi(v)&0&p\tilde I&0&0&\cdots&0&0\\
0&*&0&0&-I&0&\cdots&0&0\\
0&*&0&0&*&-I&\cdots&0&0\\
\vdots&\vdots&\vdots&\vdots&\vdots&\vdots&\ddots&\vdots&\vdots\\
0&*&0&0&*&*&\cdots&-I&0\\
0&*&-e_2&\chi\phi(\tilde\M)&*&*&\cdots&*&-I
\end{matrix}\right).\]
The determinant is
\[-(-1)^{(p-1)m}\frac{\phi(b)^{\tilde m}}{\phi(b)-1}p^m=\frac{\phi(b)^{\tilde m}}{1-\phi(b)}p^m.\]

\noindent
{\bf Case 3:} $\chi\neq 1$ and any $\phi$.\\
The matrix $ (\chi\phi)( A_{\tilde\theta})$ is here given by
\[\left(\begin{matrix}
0&\chi(a)-1&0&0&0&\cdots&0&0\\
\frac{1}{\chi(a)-1}& 1-\phi(b)&0&0&0&\cdots&0&0\\
0&\chi\phi(v)&0&(\chi(a)-1)I&0&\cdots&0&0\\
0&*&0&-I&(\chi(a)-1)I&\cdots&0&0\\
0&*&0&*&-I&\cdots&0&0\\
\vdots&\vdots&\vdots&\vdots&\vdots&\ddots&\vdots&\vdots\\
0&*&0&*&*&\cdots&-I&(\chi(a)-1)I\\
0&*&\chi\phi(\M)&*&*&\cdots&*&-I
\end{matrix}\right).\]
Using Lemma \ref{detM} we compute for the determinant
\[
-(-1)^{(p-1)m^2}\det(\chi\phi(\M))(\chi(a)-1)^{m(p-1)}=(-1)^{m+1}\phi(b)^{\tilde m-1}(\chi(a)-1)^{m(p-1)}.
\]
This concludes the proof of Proposition \ref{representative prop}.
\end{proof}
\end{subsection}
\end{section}

\begin{section}{The computation of $T_{N/K}  - [\calL, \rho_N, H_N]$}\label{section_normresolvent_gauss}

In this section we compute a representative of  $T_{N/K}  - [\calL, \rho_N, H_N]$ 
in $K_0(\ZpG, \Qpc) \simeq \Q_p^c[G]^\times / \ZpG^\times$.
The individual terms are described in Section \ref{shape}.

\begin{subsection}{Norm resolvents and Gau\ss\ sums}

If $L/K$ is a finite abelian extension of $p$-adic fields with Galois group $H$ and $\beta \in L$ a normal basis element for
$L/K$, i.e $L = K[H]\beta$, then we define the resolvent of $\beta$ for every irreducible character $\chi$ of $H$ by
\[
(\beta \mid \chi) := \sum_{g \in H} g(\beta)\chi(g^{-1}).
\]
The norm resolvent $\calN_{K/\Qp}(\beta \mid \chi)$ is defined by
\[
\calN_{K/\Qp}(\beta \mid \chi) := \prod_{\omega} (\beta \mid \chi^{\omega^{-1}})^\omega,
\]
where $\omega$ runs through a (right) transversal of $\gal(\Q_p^c / \Qp)$ modulo  $\gal(\Q_p^c / K)$.

For later reference we state the following lemma which is well known and easy to prove.

\begin{lemma}\label{resol lemma}

a) Let $L_2 \supseteq L_1 \supseteq K$ be a tower of extensions of finite abelian $p$-adic fields. Let $\beta \in L_2$ be a normal basis element for $L_2 / K$
and let $\chi$ be an irreducible character of $\gal(L_1/K)$. We write $\psi = \mathrm{inf}_{\gal(L_1/K)}^{\gal(L_2/K)}(\chi)$ for the inflation of $\chi$. Then
\[
(\beta \mid \psi ) = (\Tr_{L_2/L_1}(\beta) \mid \chi).
\]

b) Let $L_1/K$ and $L_2/K$ be finite abelian extensions of $p$-adic fields such that $L_1 \cap L_2 = K$. Let
$\beta_1$ and $\beta_2$ be normal basis elements for $L_1$ and $L_2$, respectively. We write each irreducible character $\chi$ of
$\gal(L_1L_2 / K)$ in the form $\chi = \chi_1 \chi_2$ with irreducible characters of $\gal(L_1/K)$ and $\gal(L_2 /K)$.
Then $\beta := \beta_1 \beta_2$ is a normal basis element for $L_1 L_2 / K$ and
\[
(\beta \mid \chi) = (\beta_1 \mid \chi_1)  (\beta_2 \mid \chi_2).
\]
\end{lemma}
\begin{proof}
Easy verification.
\end{proof}

Given an extension $L/K$ of local fields and an irreducible character $\psi$ of $\gal(L/K)$ 
we will use the short notation $\psi(\alpha)$ to denote $\psi((\alpha,L/K))$ 
where  $(\alpha,L/K)$ is the Artin symbol for $\alpha\in K^\times$.

\begin{lemma}\label{inflationgauss}
  Let $L/K$ be a finite abelian wildly and weakly ramified extension with group $H$. Suppose that $K/\Qp$ is unramified. Let $\chi, \phi$ denote
irreducible characters of $H$ and suppose that $\phi$ is unramified. Then
\[
\tau_K(\phi) = 1 \qquad\text{ and }\qquad \tau_K(\phi\chi) = \phi(p^{-2}) \tau_K(\chi).
\]
\end{lemma}
\begin{proof}
This is a simple reformulation of \cite[Prop.~3.8]{PickettVinatier}. 
If $K/\Qp$ is an arbitrary finite extension, then we let $D_K = \pi_K^s\OK$ denote the absolute different of $K/\Qp$.
Then $\tau_K(\phi) = \phi(\pi_K^{-s})$ by the definition of $s$ and local Galois Gau\ss\ sums. If $K/\Qp$ is unramified, then
$s=0$ and we obtain the first equality. 
The second equality is the last displayed equality in the proof of \cite[Prop.~3.8]{PickettVinatier} with $s=0$ and $\pi_K=p$.
\end{proof}

Following the arguments of \cite[bottom of page 1188]{PickettVinatier} we apply Corollary 3.4 of loc.cit. with $\pi = p$. 
So there exist extensions $\tilde M$ and $\tilde K'$ 
such that $\tilde M /K$ is a weakly and wildly ramified extension of degree $p$, the extension $\tilde K'/K$ is unramified 
and such that we have a diagram of the form
\[
\xymatrix{
& \tilde N \ar@{-}[dl] \ar@{-}[d] \ar@{-}[dr] &  \\
\tilde M \ar@{-}[dr]  & N \ar@{-}[d]  & \tilde K' \ar@{-}[dl] \\
&K&
 }\]
Moreover we may assume that $\tilde N /N$ is unramified.
By \cite[V, Cor.~(5.6)]{Neukirch92} $p$ belongs to the norm group $\calN_{\tilde M/K}(\tilde M^\times)$, 
so that we can apply  \cite[Th.~2]{PickettVinatier}.

\begin{lemma}\label{PV Th 2 lemma}
  There exist a normal basis generator $\alpha_{\tilde M}$ of the square root of the inverse different of $\tilde M / K$ and choices in the definitions of the norm resolvents
such that for all irreducible characters $\tilde \chi$ of $\gal(\tilde M / K)$ we have 
\[
\frac{\calN_{K/\Qp}(\alpha_{\tilde M} \mid \tilde \chi)}{\tau_K(\tilde\chi)} = 
\begin{cases} 1, & \tilde\chi = \tilde\chi_0, \\ p^{-m} \tilde\chi(4), & \tilde\chi \ne \tilde\chi_0.
\end{cases}
\]
\end{lemma}

\begin{proof}
By \cite[Th.~2]{PickettVinatier} and using the notation of loc. cit. we may assume that 
\begin{equation}\label{abc}
\mathcal N_{K/\Q_p}(\alpha_{\tilde M}|\tilde\chi)\tau_K^\star(\tilde\chi-\tilde\chi^2) = 1.
\end{equation}
The proof of Lemma \ref{PV Th 2 lemma} now follows immediately from the definition of $\tau_K^\star$. In a little more detail,
if $ \tilde\chi = \tilde\chi_0$ is the trivial character, then $\tau_K^\star(\chi - \chi^2) = 1$ and also $\tau_K(\tilde\chi_0) = 1$. For $\tilde\chi \ne \tilde\chi_0$
we have 
\[
\tau_K^\star(\tilde\chi-\tilde\chi^2)=\tilde\chi\left(\frac{c_{\tilde\chi}}{4c_{K,2}}\right)\psi_K(c_{\tilde\chi}^{-1})^{-1},
\]
by \cite[Prop.~3.9]{PickettVinatier}. Furthermore,
$\tau_K(\tilde\chi)=p^m\tilde\chi(c_{\tilde\chi}^{-1})\psi_K(c_{\tilde\chi}^{-1})$,
by the last displayed formula in the proof of \cite[Prop.~3.9]{PickettVinatier}. In our case we can choose $c_{K,2} = p^2$ and since $p$ is in the norm
group of $\tilde M/K$ we have $\tilde\chi(c_{K,2}) = 1$. Hence we obtain $\tau_K^\star(\tilde\chi-\tilde\chi^2) = p^m \tau_K(\tilde\chi)^{-1}\tilde\chi(4)^{-1}$.
The result now follows from (\ref{abc}).
\end{proof}

We now fix $\alpha_{\tilde M}$ as in Lemma \ref{PV Th 2 lemma}. Choose an integral normal basis element $\tilde\theta_2$ of $\tilde K'/K$ such that
$\Tr_{\tilde K'/K}(\tilde\theta_2) = 1$ and set
\[
\alpha_M := \Tr_{\tilde N/M}( \alpha_{\tilde M} \tilde\theta_2).
\]
It is easy to verify that $\alpha_M$ is an $\OKG$-generator of the square root of the inverse different of $M/K$.

\begin{lemma}\label{PickettVinatierstrong}
  Let $\chi$ be an irreducible character of $\gal(M/K)$. Then
\[\frac{\mathcal N_{K/\Q_p}(\alpha_{M}|\chi)}{\tau_K(\chi)}=\begin{cases}1&\text{if $\chi=\chi_0$}\\
p^{-m}\chi(4)&\text{if $\chi\neq\chi_0$.}\end{cases}\]
\end{lemma}

 \begin{proof}
 We write $\mathrm{inf}$ for $\mathrm{inf}_{\gal(M/K)}^{\gal(\tilde N /K)}$. Since $\tilde N/M$ is unramified we see that
for each irreducible character
 $\chi$ of $\gal(M/K)$ we obtain $\mathrm{inf}(\chi) = \tilde\chi \tilde\phi_0$, where $\tilde\phi_0$ is the trivial character
 of $\gal(\tilde K' / K)$ and $\tilde\chi$ is a uniquely determined irreducible character of $\gal(\tilde M /K)$.
Moreover, $\chi = \chi_0$ if and only if  $\tilde\chi = \tilde\chi_0$.

 By Lemma \ref{resol lemma} we have
\[
(\alpha_M \mid \chi) = (\alpha_{\tilde M}\tilde\theta_2 \mid \mathrm{inf}(\chi) ) = (\alpha_{\tilde M}\mid \tilde\chi) (\tilde\theta_2 \mid \tilde\phi_0)
= (\alpha_{\tilde M}\mid \tilde\chi)
\]
because $(\tilde\theta_2 \mid \tilde\phi_0) = \Tr_{\tilde K' /K}\tilde\theta_2 = 1$. Recall that local Galois Gau\ss\ sums are invariant under inflation of characters (see e.g. \cite[(5)]{PickettVinatier} and \cite[p. 18]{Martinet77}). We therefore get from
$\mathrm{inf}(\chi) = \mathrm{inf}_{\gal(\tilde M/K)}^{\gal(\tilde N /K)} (\tilde \chi)$
\[
\frac{\mathcal N_{K/\Q_p}(\alpha_{M}|\chi)}{\tau_K(\chi)} =
\frac{\mathcal N_{K/\Q_p}(\alpha_{\tilde M}\theta_2 | \mathrm{inf}\chi)}{\tau_K(\mathrm{inf}\chi)}
=\frac{\mathcal N_{K/\Q_p}(\alpha_{\tilde M} | \tilde\chi)}{\tau_K(\tilde\chi)}.
\]

To conclude by Lemma \ref{PV Th 2 lemma} we notice that 
\[
\tilde\chi(4) = \mathrm{inf}_{\gal(\tilde M / K)}^{\gal(\tilde N / K)}(\tilde\chi)((4|\tilde N/K))=
\mathrm{inf}_{\gal(M / K)}^{\gal(\tilde N / K)}(\chi)((4|\tilde N/K)) = \chi(4).
\]

\end{proof}

\begin{prop}\label{gaussnormresolvent}
We assume the setting introduced in Section \ref{shape}. 
Let $\psi = \chi\phi$ be a character of $G$. Then
\[\frac{\mathcal N_{K/\Q_p}(p^2\alpha_{M}\theta_2|\chi\phi)}{\tau_K(\phi\chi)}=\begin{cases}p^{2m}\mathcal 
N_{K/\Q_p}(\theta_2|\phi)&\text{if $\chi=\chi_0$}\\p^{m}\chi(4)\mathcal N_{K/\Q_p}(\theta_2|\phi)\phi(p^{2})&\text{if $\chi\neq\chi_0$,}\end{cases}\]
\end{prop}

\begin{proof}
The result follows from Lemma \ref{resol lemma}, Lemma \ref{inflationgauss} and Lemma \ref{PickettVinatierstrong}.
\end{proof}

\end{subsection}

\begin{subsection}{A representative for $T_{N/K} - [\calL, \rho_n, H_N]$}

In the following proposition we describe a representative in $\Q_p^c[G]^\times$ for the element $T_{N/K} - [\calL, \rho_n, H_N]$.

\begin{prop}\label{representative prop 2}
We assume the setting introduced in Section \ref{shape}. 
For $\calL = \p_N^{p+1}$ the element 
$T_{N/K} - [\calL, \rho_n, H_N]$ 
is represented by $\eta \in \Q_p^c[G]^\times$ where 
\[
(\eta)_{\chi\phi}=\begin{cases}p^{-2m}\mathcal N_{K/\Q_p}(\theta_2|\phi)^{-1}\delta_K^{-1}&\text{if $\chi=\chi_0$}\\
             p^{-m}\chi(4)^{-1}\mathcal N_{K/\Q_p}(\theta_2|\phi)^{-1}\phi(b)^{2}\delta_K^{-1}&\text{if $\chi\neq\chi_0$,}
\end{cases}
\]
 where $\delta_K$ is a square root of the discriminant of $K$.
\end{prop}

\begin{proof}
Recall that $\calL = \OKG (p^2 \alpha_M \theta_2)$.
As already explained in Section \ref{shape} the element $[\calL, \rho_n, H_N]$ is then represented by 
$\left( (\delta_K \calN_{K/\Qp} (p^2 \alpha_M \theta_2 \mid \chi\phi) \right)_{\chi, \phi}$. 

By definition the term $T_{N/K}$ is represented by $\left( \tau_\Qp(i_K^\Qp(\chi\phi) \right)_{\chi, \phi}$. 
Since Gau\ss\  sums are inductive in degree zero and $\tau_K(\phi) = 1$ for unramified characters by Lemma \ref{inflationgauss} 
we have
\[
\tau_{\Q_p}(i^{\Q_p}_K\chi\phi)=\tau_{\Q_p}(i^{\Q_p}_K(\chi\phi-\chi_0\phi_0))\tau_{\Q_p}(i^{\Q_p}_K\chi_0\phi_0)
=\tau_K(\chi\phi)\tau_{\Q_p}(i^{\Q_p}_K\chi_0\phi_0).
\]
Since $K/\Qp$ is unramified, $i^{\Q_p}_K(\chi_0\phi_0)$ is a sum of unramified characters so that 
$\tau_{\Q_p}(i^{\Q_p}_K\chi_0\phi_0) = 1$  and we obtain $\tau_{\Q_p}(i^{\Q_p}_K\chi\phi) = \tau_K(\chi\phi)$.
Furthermore,
$\phi(p^2)=\phi((p^2,K'/K))=\phi(F^2)=\phi(b)^{-2}$ by \cite[XIII, \S 4, Prop.~13]{SerreLocalFields} and the definition of $b$.
Combining these observations with Proposition \ref{gaussnormresolvent} concludes the proof of the proposition.
\end{proof}

\end{subsection}
\end{section}

\begin{section}{Proof of Theorem \ref{main theorem}}\label{section_the_proof}

In our setting the correction term $M_{N/K}$ is explicitly given by
\[
M_{N/K}=\frac{{}^*\left( de_G \right){}^*\left( (1-b^{-1}q^{-1})e_I\right)}{{}^*\left((1-b)e_I \right)}.
\]
It is represented by $m = m_{N/K}$ where
\[
m_{\chi\phi}=\begin{cases}d(1-q^{-1})&\text{if $\chi=\chi_0$ and $\phi=\phi_0$}\\
\frac{1-\phi(b)^{-1}q^{-1}}{1-\phi(b)}&\text{if $\chi=\chi_0$ and $\phi\neq\phi_0$}\\
1&\text{if $\chi\neq\chi_0$.}
\end{cases}
\]
As explained in Section \ref{plan} we must show that a representative of $T_{N/K} + C_{N/K} -  M_{N/K}$ 
lies in $\Opt[G]^\times$. 

Combining the results of the previous sections we see that  $T_{N/K} + C_{N/K} -  M_{N/K}$ 
is represented by an element $\omega \in \Q_p^c[G]^\times$ where $\omega = \epsilon\eta / m_{}$.
Let $W_{\theta_2}\in\Opt[G]$ be such that $\chi\phi(W_{\theta_2})=\mathcal N_{K/\Q_p}(\theta_2|\phi)\delta_K$. Then
\[
\begin{split}\omega_{\chi\phi}
&=\begin{cases}
\frac{dp^m}{p^{2m}d(1-q^{-1})}\cdot\frac{1}{\chi\phi(W_{\theta_2})}&\text{if $\chi=\chi_0$ and $\phi=\phi_0$}\\
\frac{\phi(b)^{\tilde m}p^m(1-\phi(b))}{(1-\phi(b))p^{2m}(1-\phi(b)^{-1}q^{-1})}\cdot\frac{1}{\chi\phi(W_{\theta_2})}&\text{if $\chi=\chi_0$ and $\phi\neq\phi_0$}\\
\frac{(-1)^{m+1}\phi(b)^{\tilde m-1}(\chi(a)-1)^{m(p-1)}}{p^{m}\chi(4)\phi(b)^{-2}}\cdot\frac{1}{\chi\phi(W_{\theta_2})}&\text{if $\chi\neq\chi_0$}
\end{cases}\\
&=\begin{cases}
\frac{1}{p^m-1}\cdot\frac{1}{\chi\phi(W_{\theta_2})}&\text{if $\chi=\chi_0$ and $\phi=\phi_0$}\\
\frac{\phi(b)^{\tilde m+1}}{\phi(b)p^m-1}\cdot\frac{1}{\chi\phi(W_{\theta_2})}&\text{if $\chi=\chi_0$ and $\phi\neq\phi_0$}\\
(-1)^{m+1}\frac{\phi(b)^{\tilde m+1}}{\chi(4)}\cdot\left(\frac{(\chi(a)-1)^{p-1}}{p}\right)^m\cdot\frac{1}{\chi\phi(W_{\theta_2})}&\text{if $\chi\neq\chi_0$}
\end{cases}\\
&=\begin{cases}
\frac{\phi(b)^{\tilde m+1}}{\phi(b)p^m-1}\cdot\frac{1}{\chi\phi(W_{\theta_2})}&\text{if $\chi=\chi_0$}\\
(-1)^{m+1}\frac{\phi(b)^{\tilde m+1}}{\chi(4)}\cdot\left(\frac{(\chi(a)-1)^{p-1}}{p}\right)^m\cdot\frac{1}{\chi\phi(W_{\theta_2})}&\text{if $\chi\neq\chi_0$}.
\end{cases}
\end{split}\]
We can easily write $\omega$ as an element of $\Q_p^c[G]^\times$,
\[
\omega=\frac{1}{W_{\theta_2}}\left(\frac{b^{\tilde m+1}}{bq-1}e_a+(-1)^{m+1}b^{\tilde m+1}\sigma_4^{-1}\left(\frac{(a-1)^{p-1}}{p}\right)^{m}(1-e_a)\right),
\]
where $\sigma_4=(4, M/K)\in\gal(M/K)=\langle a\rangle\subseteq G$. 
We have to prove that $\omega \in \Opt[G]^\times$.

Since $K/\Q_p$ is unramified, we know that $\delta_K$ is a unit in $\Opt[G]^\times$. 
Then by \cite[Sec.~I, Prop.~4.3]{Froehlich83} the same is true  for $W_{\theta_2}$. 
Since also $bq-1$ is clearly a unit, we can study 
$\tilde \omega=W_{\theta_2}(bq-1)\omega$
instead of $\omega$. We have
\[\tilde \omega =b^{\tilde m+1}e_a-b^{\tilde m+1}\sigma_4^{-1}\left(-\frac{(a-1)^{p-1}}{p}\right)^{m}(bq-1)(1-e_a).\]
We first show that $\tilde\omega$ is contained in $\Opt[G]$.
To that end it is enough to show that the coefficient of  $b^j$ for all $j$ is contained  in $\Opt[a]$. 
The only non-zero coefficients are those of $b^{\tilde m+1}$ and $b^{\tilde m+2}$ which are, respectively,
\[e_a+\sigma_4^{-1}\left(-\frac{(a-1)^{p-1}}{p}\right)^{m}(1-e_a).\]
and
\[-\sigma_4^{-1}\left(-\frac{(a-1)^{p-1}}{p}\right)^{m}q(1-e_a).\]
The second one has clearly coefficients in $\Zp$.
As for the first one, its integrality is equivalent to
\[1\equiv\chi(\sigma_4)^{-1}\left(-\frac{(\chi(a)-1)^{p-1}}{p}\right)^{m}\pmod{1-\zeta_p},\]
for any non-trivial character $\chi$, which follows from $\frac{(\chi(a)-1)^{p-1}}{p} \equiv -1 \pmod{1-\zeta_p}$. 

We have now shown that $\omega \in \Opt[G]$. By \cite[Cor.~3.8]{Breuning04} $\omega$ is actually a unit
in $\mathcal M^t$ where $\mathcal M^t$ denotes the maximal order in $\Q_p^t[G]$. Here $\Q_p^t = \mathrm{Quot}(\Opt)$ denotes the maximal tamely
ramified extension of $\Qp$. It follows that $\omega \in  \left( \mathcal M^t \right)^\times \cap \Opt[G] = \Opt[G]^\times$.
\qed
\end{section}

\begin{section}*{Acknowledgements}
The second named author would like to thank the Max Planck Institute for Mathematics in Bonn for its hospitality and support during the preparation of this paper.
\end{section}

\begin{tabularx}{\textwidth}{XX}
   Werner Bley & Alessandro Cobbe\\
   Math. Inst. der LMU München &    Math. Inst. der LMU München\\
   Theresienstr. 39  & Theresienstr. 39 \\
   D-80333 M\"unchen & D-80333 M\"unchen\\
   Germany & Germany\\
   bley@math.lmu.de & cobbe@math.lmu.de
\end{tabularx}

\end{document}